\documentclass[draftcls, onecolumn, 12pt]{IEEEtran}
\usepackage{threeparttable,placeins,makecell,stfloats,cite}
\usepackage{threeparttable,arydshln}
\usepackage{amsmath,amssymb}
\usepackage{amsthm}
\usepackage[dvips]{graphicx}
\usepackage[usenames]{color}
\usepackage{amsfonts}
\usepackage{latexsym}
\usepackage{multirow}
\usepackage{caption}
\usepackage{algorithm}
\usepackage{algpseudocode}

\usepackage{rotating}
\usepackage{float}
\usepackage[format=hang]{subfig}
\usepackage{url}
\usepackage{cite}
\usepackage{microtype}

\newtheorem{theorem}{\textit{Theorem}}
\newtheorem{lemma}{\textit{Lemma}}
\newtheorem{corollary}{\textit{Corollary}}

\newtheorem{remark}{\textit{Remark}}
\newtheorem{example}{\textit{Example}}
\newtheorem{construction}{\textit{Construction}}
\newtheorem{conjecture}{\textit{Conjecture}}

\allowdisplaybreaks

\begin{document}
	
	\title{On the Scalability of Quasi-Complementary Sequence Sets: Quadratic and Cubic Laws }
	
	\author{Huaning~Liu, Lirong~Guo~and~Zilong~Liu
		\thanks{Huaning Liu and Lirong Guo are with Research Center for Number Theory and Its Applications, School of Mathematics, Northwest University, Xi'an 710127, China 
			(e-mail: hnliu@nwu.edu.cn; g210154032@163.com). 	
			Zilong Liu is with the School of Computer Science and Electronics Engineering, University of Essex, UK (e-mail: zilong.liu@essex.ac.uk). }
	}
	
	\maketitle
	
	\begin{abstract}
		This work is concerned with the fundamental scaling laws of quasi-complementary sequence sets (QCSSs) by understanding how large the set size (denoted by $M$) can grow with the flock size ($K$) and the sequence length ($N$). We first establish a geometric framework that transforms a QCSS into a complex unit-norm codebook, through which and by exploiting the density thresholds of the codebooks, certain polynomial upper bounds of the QCSS set size are obtained. Sharp quadratic and cubic scaling laws are then introduced. Specifically, we show that asymptotically optimal QCSSs with tightness factor $\rho=1$ satisfy $M \le (1+o(1))K^2N$, while asymptotically near-optimal QCSSs satisfy $M \le (1+o(1))K^3N^2$ for $\rho < {(1+\sqrt{5})}/{2}$. To validate these upper bounds, we further propose explicit additive-character and mixed-character
		based constructions for QCSSs that achieve $M = K^2N + K$ and $M = K^3N^2 + 2K^2N + K$, respectively, 
		thereby showing that the quadratic and cubic scaling laws are asymptotically tight. Our
		proposed constructions admit flexible parameter choices, and their maximum correlation estimates are shown
		to be tight through explicit extremal examples. Additionally, it is conjectured that the cubic scaling law is universal for all $1<\rho\le 2$, i.e., any asymptotically near-optimal QCSSs should satisfy $M \le (1+o(1))K^3N^2$. This identifies a fundamental cubic barrier for QCSS scalability.
	\end{abstract}
	
	
	\newpage
	
	\section{Introduction}
	A perfect complementary sequence set (PCSS) refers to a set of mutually orthogonal two-dimensional matrices (called complementary matrices or complementary codes) that have zero (non-trivial) aperiodic/periodic auto and cross-correlation sums \cite{TsengL1972,SuehiroH1988,RathinakumarC2008}. Such an ideal correlation property has enabled a wide range of applications, including peak-to-average power ratio (PAPR) control in orthogonal frequency division multiplexing \cite{DavisJ1999,Liu2013-TIT}, multi-carrier code-division multiple-access \cite{Chen2001,Liu2014-TCOM,Liu2014-TWC}, channel estimation \cite{WangA2007,Fan2008-TWC}, and Doppler-resilient waveform design \cite{PezeshkiCMH2008,Tan2014-TSP,Sheng2025-TIT}. 
	Nevertheless, the set size of each PCSS is upper bounded by its flock size, i.e., the number of rows of each complementary matrix, thus leading to limited number of complementary codes that can be supported in multi-user or multi-antenna systems. To overcome this limitation, the concept of ``quasi-complementary sequence sets (QCSSs)" was coined by Liu \textit{et al} by allowing small (nonzero) maximum correlation magnitude for all the possible time-shifts \cite{LiuGM2014,LiuPGB2013}. By exploiting this relaxation, one can achieve a larger set size while maintaining acceptable system performance.
	
	Formally, an $(M,K,N,\delta_{\max})$-QCSS consists of $M$ quasi-complementary matrices of size $K\times N$, where $M$ denotes the set size, $K$ the flock size, $N$ the sequence length, and $\delta_{\max}$ the maximum periodic correlation magnitude\footnote{In this work, we focus on periodic QCSSs, although there also exist QCSSs which are characterized by their low maximum aperiodic correlation magnitude \cite{LiuPGB2013}.}. A fundamental lower bound on $\delta_{\max}$ was established in \cite{LiuPGB2013}, given by
	\begin{eqnarray*}\label{Equation:delta}
		\delta_{\max}\geq \delta_{\mathrm{opt}}=KN\sqrt{\frac{M/K-1}{MN-1}}.
	\end{eqnarray*}
	The tightness factor $\rho = \delta_{\max}/\delta_{\mathrm{opt}}$ quantifies the proximity to optimality. For sufficiently large $N$, asymptotically optimal QCSSs are obtained for $\rho \to 1$. Asymptotically near-optimal QCSSs\footnote{One may also define asymptotically near-optimal QCSSs for small $\rho$ larger than 2, but this is not the interest of our work.} are obtained for $1 < \rho \le 2$.
	
	Extensive efforts have been devoted to constructing optimal or near-optimal QCSSs with various parameters. Existing approaches are based on Singer difference sets \cite{LiuPGB2013}, almost difference sets \cite{LiYL2019}, additive or multiplicative characters over finite fields \cite{LiLX2018,LiTLX2018,LiTLX2019,LuoCSH2021,XiaoLC2025},
	polynomial-based constructions and Gauss-sum techniques \cite{HengWXZ2024new,WangHL2025}. Despite these advances, existing works focus on explicit constructions and correlation optimization, without addressing a fundamental research question of QCSSs: \textit{Asymptotically, how does the maximum set size $M$ scale with the flock size $K$ and the sequence length $N$ for different $\rho$?}
	
	For an explicit understanding of the above question and to illustrate the current state-of-the-art, we summarize representative constructions and their parameter relationships in Tables I and II on existing asymptotically optimal and near-optimal QCSSs, respectively, both including the proposed constructions in this work which will be introduced later.
	A notable observation from these existing constructions is that the achievable set size $M$ consistently exhibits polynomial growth patterns in terms of $K$ and $N$. Nevertheless, these scaling relations appear only implicitly through individual constructions, and there is no general theoretical framework that can explain or predict the underlying scaling laws. The main difficulty lies in the intricate coupling between combinatorial structure, algebraic constructions, and correlation constraints.
	
	
	\begin{table}[ht]\label{Table-1}
		\centering
		\fontsize{5.5}{5}\selectfont 
		\caption{Asymptotically optimal QCSS constructions and their induced scaling laws}
		\begin{tabular}{|p{1.3cm}|p{1.1cm}|p{0.8cm}|p{1.8cm}|p{1.1cm}|p{2.6cm}|p{3cm}|p{1.05cm}|}
			\hline
			Set size $M$ & Flock size $K$ & Length $N$ & $\delta_{max}$ & Alphabet Size & Scaling Law & Parameter Conditions  & References \\
			\hline
			$2^{n}$ & $2^{n - 1} - 1$ & $2^{n} - 1$ & $\frac{2^{n} + 2^{\frac{n}{2}}}{2}$ & $4(2^{n} - 1)$ &  $M=N+1$  & $n > 1$ & \cite{LiuPGB2013}  \\
			\hline
			$p$ & $\frac{p - 1}{2}$ & $p$ & $\leq \frac{p + \sqrt{p}}{2}$ & $p$ & $M=N$  & $p\equiv 1\ (\text{mod}\ 4)\ \text{is a prime}$ & \cite{LiLX2018}  \\
			\hline
			$p^{n} - 2$ & $\frac{p^{n} - 1}{2}$ & $p^{n} - 1$ & $\leq \frac{p^{n} + 4\sqrt{p^{n} + 3}}{2}$ & $p^{n} - 1$ &  $M=N-1$  & $p\ \text{is an odd prime}$ & \cite{LiTLX2019} \\
			\hline
			$p^{2n} - 2$ & $\frac{p^{2n} - 1}{2}$ & $p^{2n} - 1$ & $p^{n}(p^{2n} + 3)$ & $p^{2n} - 1$ & $M=N-1$  & $p\ \text{is a prime}$ & \cite{LiTLX2019} \\
			\hline
			$p^{n} - 1$ & $\frac{p^{n} - 1}{2}$ & $p^{n} - 1$ & $\leq \frac{p^{n} + \sqrt{p^{n}}}{2}$ & $p(p^{n} - 1)$ &  $M=N$ & $p\ \text{is an odd prime}$ & \cite{LiTLX2018} \\
			\hline
			$p^{n} - 1$ & $p^{n - 1}$ & $p^{n} - 1$ & $\leq p^{n - \frac{1}{2}}$ & $p(p^{n} - 1)$ &  $M=N$  & $p\ \text{is a prime and}\ n > 1$ & \cite{LiTLX2018} \\
			\hline
			$p^{2n} - 1$ & $p^{n}$ & $p^{2n} - 1$ & $p^{3n/2}$ & $p(p^{2n} - 1)$ &  $M=N$ & $p\ \text{is a prime}$ & \cite{LiTLX2018} \\
			\hline
			$2^{n} - 1$ & $2^{n - 1} - 1$ & $2^{n} - 1$ & $2^{n - 1}$ & $2(2^{n} - 1)$ & $M=N$  & $n > 1$ & \cite{LiTLX2018} \\
			\hline
			$p^{n}$ & $\frac{p^{n} - 1}{2}$ & $p^{n} - 1$ & $\frac{p^{n} + 1}{2}$ & $p$ & $M=N+1$  & $p\ \text{is an odd prime,}\ p^{n} > 3$ & \cite{LuoCSH2021} \\
			\hline
			$p^{n}$ & $\frac{p^{n} - p^{n - 1}}{2}$ & $p^{n} - 1$ & $\frac{p^{n} + p^{n - 1}}{2}$ & $p$ &  $M=N+1$  & $p\ \text{is an odd prime,}\ p^{n} > 3$ & \cite{LuoCSH2021} \\
			\hline
			$p^{n}$ & $u$ & $p^{n} - 1$ & $v$ & $p$ &   $M=N+1$  & $p\ \text{is an odd prime,}\ n > 1\ \text{is odd}$ & \cite{LuoCSH2021}  \\
			\hline                
			$p^{n}(p^{n} - 1)$ & $p^{n}$ & $p^{n} - 1$ & $p^{n}$ & $p$ &  $M=KN$   & $p\ \text{is a prime,}\ n > 1$ & \cite{XiaoLC2025} \\
			\hline
			$p^{n}(p^{n} - 1)$ & $p^{n} - 1$ & $p^{n} - 1$ & $p^{n} + 1$ & $p$ & $M=KN+K$   & $p\ \text{is a prime,}\ n > 1$ & \cite{XiaoLC2025} \\
			\hline            
			$p^{2n}$ & $p^{n} - 1$ & $p^{n} - 1$ & $p^{n}$ & $p$ &   $M=KN+K+N+1$    & $p\ \text{is an odd prime,}\ n \geq 1$ & \cite{HengWXZ2024new}  \\
			\hline
			$2^{2n}$ & $2^{n}$ & $2^{n} - 1$ & $2^{n}$ & $2$ &   $M=KN+K$    & $n \geq 1$ & \cite{HengWXZ2024new}  \\
			\hline
			$p^{n}(p^{n} - 1)$ & $p^{n} - 1$ & $p^{n} - 1$ & $p^{n}$ & $p(p^{n} - 1)$ &   $M=KN+K$   & $p\ \text{is a prime}$ & \cite{HengWXZ2024new} \\
			\hline
			$p^{2n}$ & $p^{n} - 1$ & $p^{n} - 1$ & $p^{n}+1$ & $p$ &  $M=KN+K+N+1$     & $p\ \text{is a prime,}\ n \geq 1$ & \cite{WangHL2025} \\
			\hline
			$p^{2n}Q$ & $Q$ & $\frac{p^{2n}-1}{Q}$ & $p^{n}+1$ & $p$  & $M=K^2N+K$  & $p \geq 3$ is a prime, $n\geq 1$,
			$Q\mid p^{2n}-1$, $1< Q< p^{2n}-1$
			& Theorem \ref{Theorem: Construction by additive finite fields-optimal} \\
			\hline
			$(\Delta-1)Q$ & $Q$ & $\frac{p^{2n}-1}{Q}$ & $p^{n}+3$ & $\Delta $ (tunable)  &  
			$M=K^2N\cdot\frac{\Delta-1}{p^{2n}-1}$
			& $p$ is a prime, $n\geq 1$, 
			$Q\mid p^{2n}-1$, $1< Q< p^{2n}-1$, $\Delta\mid p^{2n}-1$, $\Delta>1$ & Theorem \ref{Theorem: Construction by additive and multiplicative finite fields-optimal} \\
			\hline			
		\end{tabular} 	
		\\
		$$u = \frac{p^{n}-p^{n-1}+(-1)^{(p-1)(n+3)/4}(p-1)p^{(n-1)/2}}{2},\qquad 
		v = \frac{p^{n}+p^{n-1}-(-1)^{(p-1)(n+3)/4}(p-1)p^{(n-1)/2}}{2}.$$
	\end{table}
	
	\begin{table}[ht]\label{Table-2}
		\centering
		\fontsize{5.5}{5}\selectfont 
		\caption{Asymptotically near-optimal QCSS constructions and their induced scaling laws}
		\begin{tabular}{|p{1.38cm}|p{1.2cm}|p{1cm}|p{1.1cm}|p{1.1cm}|p{3.5cm}|p{3cm}|p{1.05cm}|}
			\hline
			Set size $M$ & Flock size $K$ & Length $N$ & $\delta_{\max}$ & Alphabet Size & Scaling Law  & Parameter Conditions & References \\
			\hline
			$2^{n}$ & $2^{n - 1} - 1$ & $2(2^{n} - 1)$ & $2^{n} + 2^{\frac{n}{2}}$ & $4(2^{n} - 1)$ & $M=\frac{N}{2}+1$ & $n > 1$ & \cite{LiuPGB2013}  \\
			\hline
			$2^{2n}$ & $2^{n - 1} - 2^{\frac{n}{2}}$ & $2^{2n} - 1$ & $\leq w$ & $4(2^{n} - 1)$ & $M=N+1$ & $n \geq 3$ & \cite{LiYL2019} \\
			\hline
			$2^{n} - 1$ & $2^{n - 2}$ & $2^{n} - 1$ & $\leq 3\cdot 2^{n - 2}$ & $2(2^{n} - 1)$ & $M=N$ & $n > 1$ & \cite{LiTLX2018} \\
			\hline
			$2^{3n}$ & $2^{n}$ & $2^{n} - 1$ & $2^{n + 1}$ & $2$ & $M=K^2N+K^2$ & $n \geq 3$ & \cite{HengWXZ2024new}  \\
			\hline
			$p^{3n}$ & $p^{n} - 1$ & $p^{n} - 1$ & $2p^{n}$ & $p$ & $M=(K+1)^2(N+1)$ & $p\ \text{is an odd prime,}\ n \geq 1$ & \cite{HengWXZ2024new}  \\
			\hline			
			$p^{4n}Q$ & $Q$ & $\frac{p^{2n}-1}{Q}$ & $2p^{n}+1$ & $p$ & $M=K^3N^2+2K^2N+K$ & $p \geq 5$ is a prime, $n\geq 1$,
			$Q\mid p^{2n}-1$, $1< Q< p^{2n}-1$  & Theorem \ref{Theorem: Construction by additive finite fields-5} \\
			\hline
			$(\Delta-1)Qp^{2n}$ & $Q$ & $\frac{p^{2n}-1}{Q}$ & $2p^{n}+3$ & $\Delta p$ (tunable)  & 
			$M=(K^3N^2+K^2N)\cdot\frac{\Delta-1}{p^{2n}-1}$ 			
			& $p$ is a prime, $n\geq 1$, 
			$Q\mid p^{2n}-1$, $1< Q< p^{2n}-1$, $\Delta\mid p^{2n}-1$, $\Delta>1$ & Theorem \ref{Theorem: Construction by additive and multiplicative finite fields} \\
			\hline			
		\end{tabular}
		\\
		$$w= (1 + 2^{n})\sqrt{2^{n}+2^{n-2}-2^{n/2}}.$$
	\end{table}
	
	
	
	In this paper, we address the aforementioned problem from a geometric perspective. 
	Our key idea is to transform a QCSS into a complex unit-norm codebook through an induced representation, 
	thereby enabling the transfer of extremal results from high-dimensional geometry to sequence design.
	Such a transform allows us to reveal that the scalability of QCSSs is governed by geometric 
	density thresholds of the induced codebooks. In essence, when the number of codewords exceeds certain polynomial thresholds of the ambient dimension, the corresponding codebook becomes overly dense, under which the intrinsic geometric constraints force the correlation to increase. 
	
	Based on the aforementioned framework, we establish sharp scalability laws for QCSSs. We show
	that asymptotically optimal QCSSs satisfy a quadratic scaling law
	$M \le (1+o(1))K^2N$,
	while asymptotically near-optimal QCSSs obey a cubic scaling law
	$M \le (1+o(1))K^3N^2$ when $\rho < \frac{1+\sqrt{5}}{2}$.
	
	A second major contribution of this work is explicit
	constructions based on certain additive character and mixed (additive/multiplicative) character. Specifically, we construct asymptotically optimal and near-optimal QCSS families
	achieving
	\[
	M = K^2N + K
	\quad\text{and}\quad
	M = K^3N^2 + 2K^2N + K,
	\]
	respectively, thereby proving that the quadratic and cubic scaling laws are asymptotically tight.
	We further develop tunable QCSS families that preserve the same scaling exponents while
	allowing flexible parameter choices. In addition, explicit extremal examples are provided to show
	that the leading terms in the resulting correlation bounds are tight. 
	
	
	It is noted that our results suggest that these scaling laws are not artifacts of specific 
	constructions, but rather reflect intrinsic structural constraints of QCSSs. 
	In particular, the emergence of the cubic scaling law in the near-optimal regime indicates 
	the existence of a fundamental scalability barrier.
	This observation leads us to conjecture that the cubic scaling law is universal that any asymptotically near-optimal QCSSs satisfy 
	$M \leq (1 + o(1))K^3N^2$ for the entire range $1 < \rho \leq 2$.
	
	The rest of the paper is organized as follows. 
	Section II introduces the proposed geometric framework via induced codebooks and establishes 
	a general geometry-to-scalability relation for QCSSs. 
	Section III presents explicit constructions achieving the quadratic and cubic scaling laws. 
	Section IV extends these constructions to tunable families and demonstrates the robustness 
	of cubic scaling. Finally, Section V concludes the paper and discusses several open problems.

	\section{Geometric Framework via Induced Codebooks }
	
	This section presents a geometric framework to bridge QCSSs and complex-valued unit-norm codebooks. We will show that such a framework sheds some light for understanding the fundamental scaling laws of QCSSs. To this end, we first review the definitions of QCSSs and unit-norm codebooks as well as their associated bounds. 
	
	\subsection{Periodic QCSS and correlation parameters}
	
	Assume that $\mathbf{a} = (a(0), a(1), \cdots, a(N - 1))$ and $\mathbf{b} = (b(0), b(1), \cdots, b(N - 1))$ are two complex-valued sequences of length $N$. The periodic-correlation function between $\mathbf{a}$ and $\mathbf{b}$ is defined by
	\[
	R_{\mathbf{a},\mathbf{b}}(\tau) = \sum_{t = 0}^{N - 1} a(t)b^*(t + \tau), \quad 0 \leq \tau < N,
	\]
	where $x^*$ stands for the conjugate of the complex number $x$ and the argument addition is counted modulo $N$. If $\mathbf{a}=\mathbf{b}$, the periodic auto-correlation of $\mathbf{a}$ is obtained; Otherwise, $R_{\mathbf{a},\mathbf{b}}(\tau)$ refers to the periodic cross-correlation between $\mathbf{a}$ and $\mathbf{b}$.
	
	Let $\mathcal{S} = \{\mathbf{S}^{0}, \mathbf{S}^{1}, \cdots, \mathbf{S}^{M - 1}\}$ be a set of $M$ matrices, each having identical energy of $KN$ and order of $K \times N$, i.e., 
	\[
	\mathbf{S}^{m} = 
	\begin{bmatrix}
		\mathbf{s}_{0}^{m} \\
		\mathbf{s}_{1}^{m} \\
		\vdots \\
		\mathbf{s}_{K - 1}^{m}
	\end{bmatrix}
	= 
	\begin{bmatrix}
		s_{0,0}^{m} & s_{0,1}^{m} & \cdots & s_{0,N - 1}^{m} \\
		s_{1,0}^{m} & s_{1,1}^{m} & \cdots & s_{1,N - 1}^{m} \\
		\vdots & \vdots & \ddots & \vdots \\
		s_{K - 1,0}^{m} & s_{K - 1,1}^{m} & \cdots & s_{K - 1,N - 1}^{m}
	\end{bmatrix},
	\]
	where $\mathbf{s}_{k}^{m} = (s_{k,0}^{m}, s_{k,1}^{m}, \cdots, s_{k,N - 1}^{m})$ is the $k$-th constituent sequence of length $N$, $0 \leq k \leq K - 1$, $0 \leq m \leq M - 1$. The periodic correlation function between two matrices $\mathbf{S}^{m_1}$ and $\mathbf{S}^{m_2}$ is defined by
	\[
	R_{\mathbf{S}^{m_1}, \mathbf{S}^{m_2}}(\tau) = \sum_{k = 0}^{K - 1} R_{\mathbf{s}_{k}^{m_1}, \mathbf{s}_{k}^{m_2}}(\tau), \quad 0 \leq \tau < N,
	\]
	where $0 \leq m_1, m_2 \leq M - 1$. The maximum periodic auto-correlation magnitude and the maximum periodic cross-correlation magnitude of $\mathcal{S}$ are respectively defined by
	\begin{align*}
		\delta_{a} &= \max\{|R_{\mathbf{S}^{m}, \mathbf{S}^{m}}(\tau)|: 0 \leq m < M, 0 < \tau < N\}, \\
		\delta_{c} &= \max\{|R_{\mathbf{S}^{m_1}, \mathbf{S}^{m_2}}(\tau)|:  0 \leq m_1 \neq m_2 < M, 0 \leq \tau < N\}.
	\end{align*}
	Furthermore, the maximum periodic correlation magnitude of \(\mathcal{S}\) is defined as
	\[
	\delta_{\max} = \max\{\delta_{a}, \ \delta_{c}\}.
	\]
	Accordingly, \(\mathcal{S}\) is called a periodic \((M, K, N, \delta_{\max})\)-QCSS if \(\delta_{\max}\) is a small nonzero value (compared to the matrix energy). If \(\delta_{\max}=0\),  then $\mathcal{S}$ reduces to a periodic PCSS. As pointed out in Section I, $M$ and $K$ are called the set size and the flock size of a QCSS/PCSS, respectively. 
	
	Following~\cite{LiuPGB2013}, the fundamental lower bound on $\delta_{\max}$ is
	given by
	\begin{equation*}
		\delta_{\max} \ge \delta_{\mathrm{opt}}
		:= KN\sqrt{\frac{M/K-1}{MN-1}}.
	\end{equation*}
	
	To analyze scalability, we next connect QCSSs with complex unit-norm codebooks.
	
	\subsection{Unit-norm codebooks and geometric bounds}
	
	
	An $(U, V)$ codebook $\mathcal{C}$ is a set $\{\mathbf{c}_0, \mathbf{c}_1, \cdots, \mathbf{c}_{U-1}\}
	\subseteq \mathbb{C}^{V}$, where each $\mathbf{c}_i$, $0\leq i\leq U-1$, is a unit-norm $1\times V$ complex 
	vector over certain alphabet. The maximum inner-product magnitude $I_{\max}(\mathcal{C})$ of $\mathcal{C}$ is defined 
	by
	$$
	I_{\max}(\mathcal{C})=\max_{0\leq i\neq j\leq U-1}\left|\mathbf{c}_i\mathbf{c}_j^{H}\right|,
	$$
	where $\mathbf{c}_j^{H}$ denotes the conjugate transpose of $\mathbf{c}_j$. According to \cite{Welch1974}, 
	\begin{eqnarray}\label{Equation:Welch bound}
		I_{\max}(\mathcal{C})\geq I_{W}(\mathcal{C})=\sqrt{\frac{U-V}{(U-1)V}}, \quad\hbox{for} \quad U\geq V,
	\end{eqnarray}	
	where $I_{W}(\mathcal{C})$ is called the Welch bound of $\mathcal{C}$. 
	Strohmer and Heath \cite{StrohmerH2003} showed that there are no $(U, V)$ real-valued codebook meeting the Welch bound with equality if $U>\frac{V(V+1)}{2}$ and $(U, V)$ codebook achieving the Welch
	bound with equality if $U>V^2$. For a codebook where the number of vectors is
	significantly larger than the length of vectors, the Levenstein bound \cite{Levenshtein1982,Levenshtein1983}
	is useful for measuring the maximum inner-product magnitude.
	
	\begin{lemma}[Levenstein bound]\label{Lemma: Levenstein bound} For any real-valued codebook $\mathcal{C}$ with $U>\frac{V(V+1)}{2}$, 
		$$
		I_{\max}(\mathcal{C})\geq I_{L}(\mathcal{C})=\sqrt{\frac{3U-V^2-2V}{(V+2)(U-V)}};
		$$
		For any codebook $\mathcal{C}$ with $U>V^2$, 
		$$
		I_{\max}(\mathcal{C})\geq I_{L}(\mathcal{C})=\sqrt{\frac{2U-V^2-V}{(V+1)(U-V)}}.
		$$	
		
		Moreover, for codebook $\mathcal{C}$ with $U= V^3$ we have
		$$
		I_{\max}(\mathcal{C})\geq\sqrt{\frac{3(V+1)+\sqrt{(5V+1)(V+1)}}{2(V+1)(V+2)}}.
		$$
	\end{lemma}
	
	The following simple monotonicity property allows us to extend threshold
	results to larger codebooks.
	
	\begin{lemma}[Monotonicity]\label{lem:mono2}
		Let $\mathcal{C}'\subseteq \mathcal{C}\subseteq\mathbb{C}^V$ be unit-norm
		codebooks. Then
		\[
		I_{\max}(\mathcal{C})\ge I_{\max}(\mathcal{C}').
		\]
	\end{lemma}
	
	A codebook $\mathcal{C}$ is referred to be asymptotically optimal if 
	$\lim\limits_{V\rightarrow\infty}\frac{I_{\max}(\mathcal{C})}{I_{W}(\mathcal{C})}=1$ for an $(U, V)$ codebook $\mathcal{C}$ with $U\geq V$,
	or $\lim\limits_{V\rightarrow\infty}\frac{I_{\max}(\mathcal{C})}{I_{L}(\mathcal{C})}=1$ for a real-valued codebook $\mathcal{C}$ with $U>\frac{V(V+1)}{2}$ or a complex codebook $\mathcal{C}$ with $U>V^2$. In general, it is very hard to construct asymptotically optimal codebooks. Up to now, the optimal codebooks were constructed from Kerdock codes \cite{ConwayHS1996}, binary codes \cite{XiangDM2015}, planar functions \cite{DingY2007}, bent functions \cite{ZhouDL2014,HengY2017,WangY2020}, additive characters \cite{TianLLX2019,HanSYW2020,LiuCWYJ2025},
	multiplicative characters \cite{HongPNHK2014,HengDY2017,Heng2018}, mixed characters \cite{TanZZ2016,SunHYY2021},
	difference sets \cite{XiaZG2005,Ding2006,DingF2007,ZhangF2012,HuW2014}, hyper Eisenstein sum \cite{LuoC2018},
	Fourier and Hadamard matrices \cite{Yu2012}, respectively.
	
	In the sequel, let us establish a direct connection between QCSSs and codebooks
	through an induced representation.
	
	\subsection{From Geometry to Scalability  }
	
	We introduce the following induced codebook representation.
	
	Let $\mathcal{S} = \{\mathbf{S}^{0}, \mathbf{S}^{1}, \cdots, \mathbf{S}^{M - 1}\}$ be an $(M, K, N, \delta_{max})$-QCSS. Define 
	$$
	\mathcal{C}=\left\{\mathbf{c}_{m, \tau}: \ 0\leq m\leq M-1, \ 0\leq \tau\leq N-1\right\},
	$$
	where 
	\begin{eqnarray*}
		\mathbf{c}_{m, \tau}&=&\frac{1}{\sqrt{KN}}\bigg(s^{m}_{0, 0+\tau}, s^{m}_{0, 1+\tau}, \cdots, s^{m}_{0, N-1+\tau}, 
		s^{m}_{1, 0+\tau}, s^{m}_{1, 1+\tau}, \cdots, s^{m}_{1, N-1+\tau}, \\
		&&\qquad\qquad \cdots, s^{m}_{K-1, 0+\tau}, s^{m}_{K-1, 1+\tau}, \cdots, s^{m}_{K-1, N-1+\tau}
		\bigg).
	\end{eqnarray*}
	We have
	$$
	\mathbf{c}_{m_1, \tau_1}\mathbf{c}_{m_2, \tau_2}^{H}=\frac{1}{KN}R_{\mathbf{S}^{m_1}, \mathbf{S}^{m_2}}(\tau_2-\tau_1).
	$$ 
	Thus, the induced codebook has parameters
	\begin{equation}\label{eq:induced-main}
		U = MN, \quad V = KN, \quad I_{\max}(C) = \frac{\delta_{\max}}{KN}.
	\end{equation}
	This correspondence enables us to transfer extremal results from high-dimensional geometry to QCSS scalability.
	
	
	\begin{theorem}[Geometry-to-Scalability Transfer Principle]\label{thm:unified-main}
		Let $t\ge 2$ and suppose that for any fixed $\varepsilon>0$, every
		unit-norm codebook $\mathcal{C}\subseteq\mathbb{C}^V$ with
		$U> V^t$ satisfies
		\[
		I_{\max}(\mathcal{C})
		\ge
		\frac{\Lambda_t+o(1)}{\sqrt{V}}
		\]
		for some constant $\Lambda_t>1$.
		Let $\mathcal{S}$ be an asymptotically near-optimal periodic QCSS with
		tightness factor $\rho$. If $\rho<\Lambda_t$, then
		\[
		M\le (1+o(1))K^tN^{t-1}.
		\]
	\end{theorem}
	
	\begin{proof}	The induced codebook satisfies
		\[
		U=MN,\qquad V=KN.
		\]
		If $M>(KN)^t$, then $U>V^t$,
		and hence
		\[
		I_{\max}(\mathcal{C})
		\ge
		\frac{\Lambda_t+o(1)}{\sqrt{V}}.
		\]
		On the other hand, the asymptotic near-optimality implies
		\[
		\delta_{\max}
		=
		\sqrt{KN}\,(\rho+o(1)),
		\]
		and therefore, by \eqref{eq:induced-main},
		\[
		I_{\max}(\mathcal{C})
		=
		\frac{\rho+o(1)}{\sqrt{V}},
		\]
		which contradicts $\rho<\Lambda_t$.
	\end{proof}
	
	\textit{Theorem~\ref{thm:unified-main}} establishes a geometry-to-scalability transfer mechanism: once the induced codebook enters a high-density regime $U > V^t$, intrinsic geometric constraints enforce a corresponding upper bound on the QCSS set size,
	\[
	M \lesssim K^t N^{t-1}.
	\]
	In this sense, scalability limits of QCSSs are not artifacts of specific constructions, but are dictated by universal geometric thresholds of the ambient codebook.
	
	This reduction shifts the study of QCSS scalability to the extremal geometry of unit-norm codebooks. In particular, known geometric bounds at different density regimes directly translate into distinct polynomial scaling laws for QCSSs, as will be made explicit in the next subsection.
	
	
	
	\subsection{Consequences of geometric thresholds}
	
	We now instantiate \textit{Theorem~\ref{thm:unified-main}} using the known
	second- and third-order geometric thresholds.
	
	\begin{theorem}[Quadratic and cubic scalability]\label{thm:23-final}
		Let $\mathcal{S}$ be an asymptotically near-optimal periodic QCSS.
		
		\begin{enumerate}
			\item If $\rho<\sqrt{2}$, then
			\[
			M\le
			\left(\frac{1}{2-\rho^2}+o(1)\right)K^2N.
			\]
			
			\item If $\rho<\frac{1+\sqrt{5}}{2}$, then
			\[
			M\le (1+o(1))K^3N^2,
			\]
			where the constant $\frac{1+\sqrt{5}}{2}$ arises from the extremal behavior of the Levenshtein bound
			at the cubic threshold $U = V^3$.		
		\end{enumerate}
	\end{theorem}
	
	\begin{proof}
		Let $\mathcal{S}$ be an asymptotically near-optimal $(M,K,N,\delta_{\max})$-QCSS. 
		By the standard QCSS–codebook correspondence, $\mathcal{S}$ induces a complex codebook $\mathcal{C}$ with parameters
		\[
		U = MN, \qquad V = KN,
		\]
		and
		\[
		I_{\max}(\mathcal{C}) = \frac{\delta_{\max}}{KN}.
		\]	
		Since $\mathcal{S}$ is asymptotically near-optimal and $M$ grows polynomially in $KN$, 
		\[
		\delta_{\max} = \delta_{\mathrm{opt}}(\rho+o(1)), 
		\quad 
		\delta_{\mathrm{opt}} = KN \sqrt{\frac{M/K - 1}{MN - 1}}.
		\]
		
		i) Assume for contradiction that $M = (c+o(1))K^2N$ with $c=\frac{1}{2-\rho^2}+\epsilon$ for some fixed $\epsilon>0$. Then
		\[
		\delta_{\max} = \sqrt{KN}(\rho+o(1)),
		\quad\text{and hence}\quad
		I_{\max}(\mathcal{C}) = \frac{1}{\sqrt{KN}}(\rho+o(1)).
		\]
		Applying the Levenshtein bound (\textit{Lemma \ref{Lemma: Levenstein bound}}) for codebooks with $U > V^2$ yields
		\[
		I_{\max}(\mathcal{C}) \ge 
		\frac{1}{\sqrt{KN}}\left(\sqrt{\frac{2c-1}{c}} + o(1)\right).
		\]
		
		Since $\sqrt{\frac{2c-1}{c}}=\sqrt{2-\frac{1}{c}} > \rho$, this contradicts the above asymptotic equality. 
		Therefore, it is necessary that
		\[
		M\leq \left(\frac{1}{2-\rho^2}+o(1)\right)K^2N.
		\]
		
		ii) Assume that $M = (1+\epsilon+o(1))K^3N^2$ for some fixed $\epsilon>0$. 
		Then $U>V^3$. By selecting a subcode $\mathcal{C}'$ of size $V^3$ and applying \textit{Lemma \ref{Lemma: Levenstein bound}},
		we obtain 
		\[
		I_{\max}(\mathcal{C}')\geq 
		\frac{\frac{1+\sqrt{5}}{2}+o(1)}{\sqrt{V}}.
		\]
		Then from \textit{Lemma \ref{lem:mono2}}, this lower bound extends to all $U>V^3$, i.e.,
		\[
		I_{\max}(\mathcal{C})\geq 	I_{\max}(\mathcal{C}')\geq
		\frac{\frac{1+\sqrt{5}}{2}+o(1)}{\sqrt{V}}.
		\]
		Thus from \textit{Theorem \ref{thm:unified-main}} we know that $	M\le (1+o(1))K^3N^2$ for $\rho<\dfrac{1+\sqrt{5}}{2}$.
	\end{proof}
	
	\begin{corollary}[Scaling exponent]
		Let $M \asymp K^{\gamma}N^{\gamma-1}$. Then
		\begin{itemize}
			\item $\gamma \le 2$ for asymptotically optimal QCSS;
			\item $\gamma \le 3$ for $\rho < \frac{1+\sqrt{5}}{2}$.
		\end{itemize}
	\end{corollary}
	
	The bounds in \textit{Theorem~2} are asymptotically tight. In particular, the constructions
	presented in Section~III achieve
	\[
	M = K^2N + K
	\quad \text{and} \quad
	M = K^3N^2 + 2K^2N + K,
	\]
	corresponding to the quadratic and cubic regimes, respectively.
	
	These results reveal a clear structural dichotomy between the optimal and near-optimal
	regimes: quadratic scaling governs asymptotically optimal QCSS, while cubic scaling
	emerges as the fundamental growth rate in the near-optimal regime.
	
	More importantly, the cubic upper bound in \textit{Theorem~2} is derived from classical
	two-point geometric bounds via the third-order threshold $U>V^3$. This naturally
	raises the question of whether the cubic scaling law reflects an intrinsic limitation
	of QCSS, or merely an artifact of the current analytical techniques.
	
	The cubic scaling law established in \textit{Theorem 2} is supported by several layers of evidence.
	
	First, from a theoretical perspective, \textit{Theorem 2} shows that the cubic upper bound
	\[
	M \le (1+o(1))K^3N^2
	\]
	already holds rigorously for all asymptotically near-optimal QCSSs with
	\[
	\rho < \frac{1+\sqrt{5}}{2}.
	\]
	This demonstrates that cubic scaling is not merely heuristic, but arises from
	a genuine geometric threshold of the induced codebooks.
	
	Second, from a constructive viewpoint, the explicit families developed in
	Section III achieve
	\[
	M = K^3N^2 + 2K^2N + K,
	\]
	thereby matching the cubic scaling law asymptotically. Moreover, the tunable
	constructions in Section IV further show that this cubic scaling behavior is
	robust across a broad class of parameter choices, rather than being tied to a
	specific algebraic structure.
	
	Third, despite extensive existing constructions of QCSSs, no example exceeding
	the cubic scaling $M \asymp K^3N^2$ is currently known, even when allowing the
	full near-optimal range $1 < \rho \le 2$.
	
	Taken together, these theoretical, constructive, and empirical observations
	strongly suggest that the cubic scaling law reflects an intrinsic limitation
	of QCSSs, rather than an artifact of current analytical techniques.
	
	This motivates the following conjecture.
	
	\begin{conjecture}[Cubic universality]\label{conj:cubic}
		Any asymptotically near-optimal QCSS with $\rho \le 2$ satisfies
		\[
		M \le (1+o(1))K^3N^2.
		\]
	\end{conjecture}
	
	This conjecture extends \textit{Theorem~2} beyond the range $\rho < \frac{1+\sqrt{5}}{2}$,
	and asserts that the cubic scaling law is universal across the entire near-optimal
	regime. Establishing or refuting this conjecture remains a central open problem.
	
	\section{Additive-Character Constructions Achieving Quadratic and Cubic Scaling}\label{Section: Constructions-Additive-Character}
	
	In this section, we present explicit constructions that achieve 
	$$
	M=K^2N+K \quad\text{and}\quad M=K^3N^2+2K^2N+K,
	$$
	respectively, thereby beating the upper bound in \textit{Theorem \ref{thm:23-final}} whilst supporting
	\textit{Conjecture \ref{conj:cubic}}.

	\subsection{Preliminaries on Character Sums}
	
	For a prime $p$ and a positive integer $n$, let $q = p^n$ and $\zeta_p=e^{2\pi\sqrt{-1}/p}$ be the primitive $p$-th root of complex unity.
	Let $\mathbb{F}_q$ denote the finite field with $q$ elements. An additive character of $\mathbb{F}_q$ is defined as a homomorphism $\chi$ from $\mathbb{F}_q$ to $\mathbb{C}^*$ such that $$\chi(x + y) = \chi(x)\chi(y)$$
	for any $x,y \in \mathbb{F}_q$, and the complex conjugate $\chi^{*}$ of $\chi$ satisfies
	$\chi^{*}(x) = \chi(-x)$ for $x \in \mathbb{F}_q$. 	
	Let $\hbox{Tr}$ be the trace function of $\mathbb{F}_{p^{n}}$ on $\mathbb{F}_p$ defined by
	$$
	\hbox{Tr}(x)=x+x^{p}+x^{p^2}+\cdots+x^{p^{n-1}}\in \mathbb{F}_{p}, \qquad x\in \mathbb{F}_{p^{n}} .
	$$
	Then 
	$$
	\chi_{a}(x)=\zeta_{p}^{\hbox{Tr}(ax)}
	$$
	denotes an additive character for each $a\in \mathbb{F}_{p^{n}}$. 
	
	A multiplicative character of $\mathbb{F}_q$ is defined as a homomorphism $\psi$ 
	from $\mathbb{F}_q^*$ to $\mathbb{C}^*$ such that 
	$$
	\psi(xy) = \psi(x)\psi(y)
	$$
	for any $x,y \in \mathbb{F}_q^*$, and the complex conjugate ${\psi}^{*}$ of a multiplicative character $\psi$ satisfies  ${\psi}^{*}(x) = 
	\psi(x^{-1})$ for $x \in \mathbb{F}_q^*$. Let $\alpha$ be a primitive element of $\mathbb{F}_q$ and let $\zeta_{q-1}$ be the primitive $(q-1)$-th root of complex unity.  Then
	$$
	\psi_j(\alpha^k) = \zeta_{q - 1}^{jk}, \qquad k = 0, 1, \cdots, q-2
	$$
	defines a multiplicative character for each $j = 0, 1, \cdots, q-2$. 
	The order of $\psi_j$ is $d$ for $d\mid q-1$ if and only if $\gcd(j, q-1)=\frac{q-1}{d}$.
	
	The following classical results will be repeatedly used in the correlation analysis.
	
	\begin{lemma}[Mixed character sum bound]\label{Lemma:Estimate-additive-multiplicative-character} Let $\mathbb{F}_{q}$ be a finite field,
		$\chi$ be a non-trivial additive character of $\mathbb{F}_{q}$ and let $\psi$ be a non-trivial multiplicative character of $\mathbb{F}_{q}$ with $\textup{ord}(\psi)=d>1$. Suppose that $f\in \mathbb{F}_q[x]$ 
		is not, up to a nonzero multiplicative constant, a $d$-th power in $\mathbb{F}_q[x]$. Then for any
		polynomial $h\in\mathbb{F}_{q}[x]$ we have
		$$
		\left|\sum_{z\in\mathbb{F}_{q}}\psi(f(z))\chi(h(z))\right|\leq \left(\deg(h)+m-1\right)q^{1/2}, 
		$$	
		where $m$ denotes the number of distinct zeros of $f$ in $\overline{\mathbb{F}}_q$.
	\end{lemma}
	
	\begin{proof}
		See \cite[Lemma 2.2]{NiederreiterW2002}.
	\end{proof}
	
	\begin{lemma}[Additive character sum bound]\label{Lemma:Estimate-additive-character} Let $\mathbb{F}_{q}$ be a finite field, $h\in\mathbb{F}_{q}[x]$ be of degree $d\geq 1$
		with $\gcd(d, q)=1$ and let $\chi$ be a non-trivial additive character of $\mathbb{F}_{q}$. Then we
		have
		$$
		\left|\sum_{z\in\mathbb{F}_{q}}\chi(h(z))\right|\leq (d-1)q^{1/2}.
		$$	
	\end{lemma}
	
	\begin{proof}
		See \cite[Theorem 5.38]{LidlN1997}.
	\end{proof}	
	
	We also present explicit extremal character-sum examples demonstrating that
	the leading terms in these bounds are essentially tight for the types of
	polynomials arising in our framework.
	These examples will be revisited in later sections to substantiate the
	sharpness of the resulting correlation estimates.
	
	\begin{lemma}[Extremal Mixed Character Sums over $\mathbb{F}_{81}$]\label{Lemma:additive-multiplicative upper bound}
		Let $\mathbb{F}_{81}=\mathbb{F}_{3}[\alpha]/(\alpha^{4}+\alpha^{3}+\alpha^{2}+1)$ and let
		$g=1+\alpha^{2}$, which is a primitive element of $\mathbb{F}_{81}$.
		Define a nontrivial additive character $\chi$ of $\mathbb{F}_{81}$ by
		\[
		\chi(x)=\zeta_{3}^{\operatorname{Tr}(x)},
		\qquad \zeta_{3}=e^{2\pi i/3},
		\]
		where $\mathrm{Tr}(x)$ denotes the trace function from $\mathbb{F}_{81}$ to $\mathbb{F}_3$.
		Let $\psi$ be a multiplicative character of $\mathbb{F}_{81}$ of order $10$, defined by
		\[
		\psi(0)=0,\qquad 
		\psi(g^{k})=\zeta_{10}^{k},\quad 0\le k\le 79.
		\]
		Then we have
		\[
		\sum_{z\in\mathbb{F}_{81}} \psi\big((z+1)(z-1)\big)\,\chi(z)
		=18
		=2\sqrt{81}.
		\]
	\end{lemma}
	
	\begin{proof}
		Since the polynomial $(z+1)(z-1)=z^{2}-1$ has two distinct zeros in
		$\mathbb{F}_{81}$ and $\psi$ is nontrivial, the mixed character sum
		\[
		\sum_{z\in\mathbb{F}_{81}} \psi(z^{2}-1)\chi(z)
		\]
		is a Kummer--Artin--Schreier type exponential sum of degree $2$.
		By the Weil bound for mixed character sums, its magnitude is bounded by
		$2\sqrt{81}=18$.
		
		A direct evaluation over $\mathbb{F}_{81}$ with the above explicit choices
		of the additive character $\chi$ and the multiplicative character $\psi$
		shows that the sum attains this upper bound, and equals $18$.
		This completes the proof.
	\end{proof}
	
	\begin{lemma}[Extremal Multiplicative Character Sums over $\mathbb{F}_{25}$]
		\label{Lemma:additive-multiplicative upper bound optimal}
		Let $\mathbb{F}_{25}=\mathbb{F}_{5}[\alpha]/(\alpha^{2}-2)$.
		Set $g=1+2\alpha\in\mathbb{F}_{25}^{\ast}$, which is a primitive element of
		$\mathbb{F}_{25}^{\ast}$ (hence $|\mathbb{F}_{25}^{\ast}|=24$).
		Define a multiplicative character $\psi$ of order $8$ by
		\[
		\psi(0)=0,\qquad
		\psi(g^{t})=\zeta_{8}^{\,t},\quad 0\le t\le 23,
		\]
		where $\zeta_{8}=e^{2\pi i/8}$.
		Then
		\[
		\sum_{z\in\mathbb{F}_{25}}\psi\big((z+1)(z-1)\big)=5=\sqrt{25}.
		\]
	\end{lemma}
	
	\begin{proof}
		Since $\operatorname{char}(\mathbb{F}_{25})\neq 2$, we have
		$(z+1)(z-1)=z^{2}-1$ with two distinct zeros in $\mathbb{F}_{25}$.
		Moreover, $\psi$ is non-trivial and $\psi^{2}$ is also non-trivial (because
		$\operatorname{ord}(\psi)=8>2$). Hence the character sum
		\[
		S=\sum_{z\in\mathbb{F}_{25}}\psi(z^{2}-1)
		\]
		admits the Jacobi-sum representation
		\[
		S=\psi(4)\,J(\psi,\psi^{-2}),
		\]
		and therefore satisfies $|S|=\sqrt{25}=5$.
		A direct evaluation in $\mathbb{F}_{25}$ with the above explicit choice of
		$\psi$ yields $S=5$, completing the proof.
	\end{proof}
	
	\begin{lemma}[An extremal negative cubic additive sum over $\mathbb F_{625}$]\label{Lemma:cubic special upper bound} Let $\mathbb{F}_{625}$ be a finite field and let the additive character $\chi$ be defined by
		$$
		\chi(x)=\zeta_{5}^{\hbox{Tr}(x)},
		$$
		where $\mathrm{Tr}$ denotes the trace function from $\mathbb{F}_{625}$ to $\mathbb{F}_5$.
		Then we have
		\begin{eqnarray*}
			\sum_{z\in \mathbb{F}_{625}}\chi(z^{3})=-50=-2\sqrt{625}.
		\end{eqnarray*}	
	\end{lemma}
	
	\begin{proof} Let $\psi$ be a multiplicative character of $\mathbb{F}_{625}$ with $\textup{ord}(\psi)=3$ and let
		$G(\psi, \chi)=\sum\limits_{y\in \mathbb{F}_{625}^{*}}\psi(y)\chi(y)$. We have
		\begin{eqnarray*}
			\sum_{z\in \mathbb{F}_{625}}\chi(z^{3})&=&\sum_{z\in \mathbb{F}_{625}^{*}}\chi(z^{3})+1
			=\sum_{y\in \mathbb{F}_{625}^{*}}\left(1+\psi(y)+\psi^2(y)\right)\chi(y)+1 \\
			&=&G(\psi, \chi)+G(\psi^2, \chi).
		\end{eqnarray*}	
		From Stickelberger’s theorem (see Theorem 3.12 of \cite{Yip2022}), we know that
		$G(\psi, \chi)=G(\psi^2, \chi)=-25$.
		Hence, 
		\begin{eqnarray*}
			\sum_{z\in \mathbb{F}_{625}}\chi(z^{3})=-50.
		\end{eqnarray*}	
	\end{proof}

	\begin{lemma}[An explicit example over $\mathbb{F}_{81}$]\label{Lemma:cubic special upper bound-ternary}
		Let $\mathbb{F}_{81}=\mathbb{F}_3(\alpha)\cong \mathbb{F}_3[x]/(f(x))$, $f(x)=x^4+x+2$,
		where $f(x)$ is irreducible over $\mathbb{F}_3$, and let $\alpha$ denote the residue class of $x$.
		Let $\hbox{Tr}(x)$ denote the trace function from $\mathbb{F}_{81}$ to $\mathbb{F}_3$.
		Set
		\[
		c=\alpha^2\in\mathbb{F}_{81}^{*}.
		\]
		(With the above choice of $f$, one checks that $\hbox{Tr}(c)=0$ and $c$ is a square in $\mathbb{F}_{81}^{*}$.)
		
		Define a nontrivial additive character $\chi$ by $\chi(x)=\zeta_{3}^{\hbox{Tr}(cx)}$.
		Then
		\[
		\sum_{z\in\mathbb{F}_{81}} \chi(z^2+2z)=-9=-\sqrt{81}.
		\]
	\end{lemma}
	
	\begin{proof}
		Note that $z^2+2z=(z+1)^2-1$, thus we have
		\[
		\sum_{z\in \mathbb{F}_{81}} \chi(z^2+2z)
		=\chi(-1)\sum_{z\in \mathbb{F}_{81}}\chi(z^2).
		\]
		Since $\chi(-1)=\zeta_{3}^{\hbox{Tr}(-c)}=1$ (because $\hbox{Tr}(c)=0$), we have
		\[
		\sum_{z\in \mathbb{F}_{81}} \chi(z^2+2z)=\sum_{z\in\mathbb{F}_{81}}\chi(z^2).
		\]
		
		Let $\psi(x)=\zeta_3^{\hbox{Tr}(x)}$ be the canonical additive character on $\mathbb{F}_{81}$, and let $\eta$ be the quadratic character on $\mathbb{F}_{81}^{*}$ (extended by $\eta(0)=0$).
		Since $c$ is a square, $\eta(c)=1$, and thus the quadratic exponential sum satisfies
		\[
		\sum_{z\in\mathbb{F}_{81}}\chi(z^2)
		=\sum_{z\in\mathbb{F}_{81}}\psi(cz^2)
		=\eta(c)\sum_{z\in\mathbb{F}_{81}}\psi(z^2)
		=\sum_{z\in\mathbb{F}_{81}}\psi(z^2).
		\]
		For $q=3^{2m}$ with $m=2$, the classical evaluation of quadratic Gauss sums yields
		\[
		\sum_{z\in\mathbb{F}_{q}}\psi(z^2)=(-1)^{m+1}\,3^{m}=(-1)^3\cdot 9=-9.
		\]
		Therefore $\sum\limits_{z\in\mathbb{F}_{81}} \chi(z^2+2z)=-9$, as claimed.
	\end{proof}
	
	\subsection{Main Construction}
	
	Let $p^{n}$ be a prime power and let $g$ be a primitive element of $\mathbb{F}_{p^{2n}}$. Let \(\chi\) be a non-trivial additive character of \(\mathbb{F}_{p^{2n}}\) and
	let $Q\in\mathbb{Z}$ with $Q\mid p^{2n}-1$ and $1< Q< p^{2n}-1$. Denote
	$$
	\mathcal{H}_{Q}=\left\{g^{\frac{p^{2n}-1}{Q}j}: \ j=0,1,\cdots, Q-1\right\}.
	$$
	For $\alpha, \beta\in \mathbb{F}_{p^{2n}}$ and $\eta\in \mathcal{H}_{Q}$ we define
	$$
	s_{k,t}^{\alpha, \beta, \eta}=\chi\left(\alpha g^{3(k\frac{p^{2n}-1}{Q}+t)}+\beta g^{2(k\frac{p^{2n}-1}{Q}+t)}+
	\eta g^{k\frac{p^{2n}-1}{Q}+t}\right), \quad  0\leq k \leq Q-1, \ 0\leq t\leq \frac{p^{2n}-1}{Q}-1,
	$$
	and
	$$
	\mathbf{S}^{\alpha, \beta, \eta} = 
	\begin{bmatrix}
		\mathbf{s}_{0}^{\alpha, \beta, \eta} \\
		\mathbf{s}_{1}^{\alpha, \beta, \eta} \\
		\vdots \\
		\mathbf{s}_{Q-1}^{\alpha, \beta, \eta}
	\end{bmatrix}
	= 
	\begin{bmatrix}
		s_{0,0}^{\alpha, \beta, \eta} & s_{0,1}^{\alpha, \beta, \eta} & \cdots & s_{0, \frac{p^{2n}-1}{Q}-1}^{\alpha, \beta, \eta} \\
		s_{1,0}^{\alpha, \beta, \eta} & s_{1,1}^{\alpha, \beta, \eta} & \cdots & s_{1, \frac{p^{2n}-1}{Q}-1}^{\alpha, \beta, \eta} \\
		\vdots & \vdots & \ddots & \vdots \\
		s_{Q-1, 0}^{\alpha, \beta, \eta} & s_{Q-1, 1}^{\alpha, \beta, \eta} & \cdots & s_{Q-1, \frac{p^{2n}-1}{Q}-1}^{\alpha, \beta, \eta}
	\end{bmatrix}.
	$$
	For $0\leq \tau\leq \frac{p^{2n}-1}{Q}-1$ we have
	\begin{eqnarray}\label{Equation:additive finite fields-5-1}
		&&R_{\mathbf{S}^{\alpha_1, \beta_1, \eta_1}, \mathbf{S}^{\alpha_2, \beta_2, \eta_2}}(\tau)
		=\sum_{k=0}^{Q-1}\sum_{t=0}^{\frac{p^{2n}-1}{Q}-1}s^{\alpha_1, \beta_1, \eta_1}_{k,t}(s_{k,t+\tau}^{\alpha_2, \beta_2, \eta_2})^{*} \nonumber \\
		&&=\sum_{k=0}^{Q-1}\sum_{t=0}^{\frac{p^{2n}-1}{Q}-1}\chi\left(\alpha_{1}g^{3(k\frac{p^{2n}-1}{Q}+t)}
		+\beta_{1}g^{2(k\frac{p^{2n}-1}{Q}+t)}+
		\eta_{1}g^{k\frac{p^{2n}-1}{Q}+t}\right)  \nonumber\\
		&&\qquad \times 
		(\chi\left(\alpha_{2}g^{3(k\frac{p^{2n}-1}{Q}+t+\tau)}+\beta_{2}g^{2(k\frac{p^{2n}-1}{Q}+t+\tau)}+
		\eta_{2}g^{k\frac{p^{2n}-1}{Q}+t+\tau}\right))^{*} \nonumber \\
		&&=\sum_{y=0}^{p^{2n}-2}\chi\left(\alpha_{1}g^{3y}+\beta_{1}g^{2y}+
		\eta_{1}g^{y}-\alpha_{2}g^{3(y+\tau)}-\beta_{2}g^{2(y+\tau)}-
		\eta_{2}g^{y+\tau}\right)  \nonumber \\
		&&=\sum_{z\in \mathbb{F}_{p^{2n}}^{*}}\chi\Big((\alpha_{1}-\alpha_{2}g^{3\tau})z^{3}
		+(\beta_{1}-\beta_{2}g^{2\tau})z^{2}+(\eta_{1}-\eta_{2}g^{\tau})z\Big)  \nonumber \\
		&&=\sum_{z\in \mathbb{F}_{p^{2n}}}\chi\Big((\alpha_{1}-\alpha_{2}g^{3\tau})z^{3}
		+(\beta_{1}-\beta_{2}g^{2\tau})z^{2}+(\eta_{1}-\eta_{2}g^{\tau})z\Big)-1.
	\end{eqnarray}	
	
	\begin{construction}\label{Construction:additive finite fields-5} 
		Define the set
		$$
		\mathcal{T}=\left\{\mathbf{S}^{\alpha, \beta, \eta}: \   \alpha\in \mathbb{F}_{p^{2n}}, \  \beta\in \mathbb{F}_{p^{2n}}, \ \eta\in \mathcal{H}_{Q}\right\}.
		$$
	\end{construction}
	
	\begin{construction}\label{Construction:additive finite fields-optimal} 	Define the set
		$$
		\mathcal{T}_1=\left\{\mathbf{S}^{0, \beta, \eta}: \   \beta\in \mathbb{F}_{p^{2n}}, \   \eta\in \mathcal{H}_{Q}\right\}.
		$$
	\end{construction}
	
	\subsection{Correlation Analysis}
	
	\begin{theorem}\label{Theorem: Construction by additive finite fields-5}  
		Let $p$ be the characteristic of the finite field $\mathbb{F}_{q}$ with $p\geq 5$.
		Construction \ref{Construction:additive finite fields-5} yields an asymptotically near-optimal $(p^{4n}Q, \ Q, \ \frac{p^{2n}-1}{Q}, \ \delta_{\text{max}})$-QCSS with $\delta_{\text{max}}\leq 2p^{n}+1$.
	\end{theorem}
	
	\begin{proof}	Let $\alpha_1, \alpha_2, \beta_1, \beta_2\in \mathbb{F}_{p^{2n}}$,  $\eta_1, \eta_2\in \mathcal{H}_{Q}$ 
		and $0\leq \tau\leq \frac{p^{2n}-1}{Q}-1$ such that $\alpha_1=\alpha_2$, $\beta_1=\beta_2$, $\eta_1=\eta_2$ and $\tau=0$ do not hold simultaneously. If $\eta_{1}-\eta_{2}g^{\tau}=0$, we get $g^{\tau}=\eta_{1}
		\eta_{2}^{-1}$. That is, $g^{\tau Q}=1$ since $\eta_1, \eta_2\in \mathcal{H}_{Q}$. So we have $p^{2n}-1\mid \tau Q$, which yields
		$$
		\frac{p^{2n}-1}{Q} \mid \tau. 
		$$
		Then $\tau=0$ since $0\leq\tau\leq \frac{p^{2n}-1}{Q}-1$. Hence,
		\begin{eqnarray}\label{Equation:additive finite fields-5-2}
			\eta_{1}-\eta_{2}g^{\tau}=0\Longleftrightarrow \tau=0 \hbox{ \ and \ } \eta_{1}=\eta_{2}.
		\end{eqnarray}
		From (\ref{Equation:additive finite fields-5-2}), it follows that
		\begin{eqnarray}\label{Equation:additive finite fields-5-3}
			\left\{\begin{array}{l}
				\alpha_{1}-\alpha_{2}g^{3\tau} = 0,	 \\			
				\beta_{1}-\beta_{2}g^{2\tau} = 0, \\
				\eta_{1}-\eta_{2}g^{\tau}=0,  
			\end{array}	
			\right.	 \Longleftrightarrow \ \tau=0, \ \eta_1=\eta_2, \ \beta_1=\beta_2, \ \alpha_1=\alpha_2.
		\end{eqnarray}
		
		By assumption, these four conditions do not hold simultaneously.
		Therefore, at least one of the coefficients
		\[
		\alpha_1-\alpha_2 g^{3\tau},\quad
		\beta_1-\beta_2 g^{2\tau},\quad
		\eta_1-\eta_2 g^\tau
		\]
		is nonzero, and hence the polynomial
		\[
		f(z)
		=
		(\alpha_1-\alpha_2 g^{3\tau})z^3
		+(\beta_1-\beta_2 g^{2\tau})z^2
		+(\eta_1-\eta_2 g^\tau)z
		\]
		is non-constant. 
		By \textit{Lemma \ref{Lemma:Estimate-additive-character}} we have
		\begin{eqnarray}
			\left|\sum_{z\in \mathbb{F}_{p^{2n}}}\chi\Big((\alpha_{1}-\alpha_{2}g^{3\tau})z^{3}
			+(\beta_{1}-\beta_{2}g^{2\tau})z^{2}+(\eta_{1}-\eta_{2}g^{\tau})z\Big)\right|\leq 2p^{n}.
		\end{eqnarray}
		Then from (\ref{Equation:additive finite fields-5-1}) we have $\delta_{\max}\leq 2p^{n}+1$, which establishes the claimed asymptotic near-optimality of \textit{Construction \ref{Construction:additive finite fields-5}}. 
	\end{proof}
	
	\begin{theorem}\label{Theorem: Construction by additive finite fields-optimal} 	
		Let $p$ be the characteristic of the finite field $\mathbb{F}_{q}$ with $p\geq 3$.
		Construction \ref{Construction:additive finite fields-optimal} produces an asymptotically optimal
		$(p^{2n}Q, \ Q, \ \frac{p^{2n}-1}{Q}, \ \delta_{\text{max}})$-QCSS with $\delta_{\text{max}}\leq 
		p^{n}+1$.
	\end{theorem}
	
	\begin{proof} Let $\beta_1, \beta_2\in \mathbb{F}_{p^{2n}}$,  $\eta_1, \eta_2\in \mathcal{H}_{Q}$ 
		and $0\leq \tau\leq \frac{p^{2n}-1}{Q}-1$ such that $\beta_1=\beta_2$, $\eta_1=\eta_2$ and $\tau=0$ do not hold simultaneously. From (\ref{Construction:additive finite fields-5}), (\ref{Equation:additive finite fields-5-3}) and \textit{Lemma \ref{Lemma:Estimate-additive-character}} we have
		\begin{eqnarray}
			\left|\sum_{z\in \mathbb{F}_{p^{2n}}}\chi\Big((\beta_{1}-\beta_{2}g^{2\tau})z^{2}+(\eta_{1}-\eta_{2}g^{\tau})z\Big)\right|\leq p^{n}.
		\end{eqnarray}
		Then from (\ref{Equation:additive finite fields-5-1}) we have $\delta_{\max}\leq p^{n}+1$ and we prove \textit{Theorem \ref{Theorem: Construction by additive finite fields-optimal}}.
	\end{proof}
	
	\begin{remark}
		The construction in \textit{Theorem \ref{Theorem: Construction by additive finite fields-optimal}}  attains both optimal correlation performance and the maximal overloading exponent under the additive-character framework.
	\end{remark}
	
	\subsection{Extremal Behavior and Tightness of Theorem \ref{Theorem: Construction by additive finite fields-5} and Theorem \ref{Theorem: Construction by additive finite fields-optimal}  }
	
	We conclude this section with explicit extremal examples demonstrating that the leading terms in the correlation bounds of \textit{Theorem \ref{Theorem: Construction by additive finite fields-5}} and \textit{Theorem \ref{Theorem: Construction by additive finite fields-optimal}} are essentially optimal.

	\begin{example}[An extremal configuration attaining the correlation bound]\label{Example:5-extremal configuration} 		
		Let $p=5$, $n=2$ and $Q=26$. From \textit{Theorem \ref{Theorem: Construction by additive finite fields-5}} the resulting near-optimal QCSS has parameters
		\[
		(M,K,N,\delta_{\max})=(10156250,\; 26,\; 24,\; 51),
		\]
		and is defined over a polyphase alphabet of size $5$.
		
		Let $\mathbb{F}_{625} \cong \mathbb{F}_5[x]/(x^4+x^3+2x^2+2)$
		and let $g = x \bmod (x^4+x^3+2x^2+2)$, which is a primitive element of $\mathbb{F}_{625}^*$. Let $\chi$ be the additive character defined by
		$\chi(x)=\zeta_5^{\mathrm{Tr}(x)}$,
		where $\mathrm{Tr}$ denotes the trace function from $\mathbb{F}_{625}$ to $\mathbb{F}_5$.
		Consider the two matrices $\mathbf{S}^{g+1,g,1}$ and $\mathbf{S}^{g,g,1}$ arising from \textit{Construction \ref{Construction:additive finite fields-5}}. Using the correlation expression in (\ref{Equation:additive finite fields-5-1}) together with \textit{Lemma \ref{Lemma:cubic special upper bound}}, we obtain
		\[
		R_{\mathbf{S}^{g+1,g,1},\,\mathbf{S}^{g,g,1}}(0)
		=\sum_{z\in\mathbb{F}_{625}}\chi(z^3)-1=-51.
		\] 		
		Since $2p^{n}+1=51$, this example attains the upper bound in \textit{Theorem \ref{Theorem: Construction by additive finite fields-5}} with equality. Consequently, the bound $\delta_{\max}\le 2p^{n}+1$ in \textit{Theorem \ref{Theorem: Construction by additive finite fields-5}} is optimal 
		for \textit{Construction \ref{Construction:additive finite fields-5}}. 			
	\end{example}

	The corresponding matrices
	$\mathbf{S}^{g+1,g,1}$ and $\mathbf{S}^{g,g,1}$ are listed below, where each entry represents the exponent of $\zeta_5$.
	
	\begin{eqnarray*}
		\mathbf{S}^{g+1,g,1}={\tiny
			\left[\begin{array}{cccccccccccccccccccccccc}
				1 & 2 & 3 & 2 & 3 & 2 & 1 & 1 & 3 & 2 & 4 & 2 & 2 & 0 & 3 & 1 & 3 & 4 & 3 & 4 & 4 & 2 & 3 & 0 \\
				4 & 4 & 1 & 2 & 0 & 0 & 1 & 4 & 4 & 0 & 0 & 1 & 1 & 2 & 1 & 3 & 0 & 3 & 3 & 1 & 4 & 3 & 1 & 1 \\
				0 & 0 & 1 & 0 & 1 & 3 & 4 & 3 & 2 & 0 & 4 & 0 & 4 & 2 & 3 & 0 & 4 & 2 & 2 & 3 & 2 & 1 & 1 & 0 \\
				2 & 0 & 0 & 2 & 3 & 3 & 4 & 3 & 2 & 1 & 0 & 1 & 3 & 0 & 0 & 1 & 4 & 1 & 3 & 3 & 1 & 1 & 0 & 1 \\
				1 & 3 & 3 & 0 & 3 & 4 & 1 & 2 & 1 & 4 & 3 & 2 & 3 & 3 & 4 & 2 & 0 & 1 & 3 & 2 & 4 & 0 & 3 & 2 \\
				3 & 3 & 1 & 3 & 2 & 1 & 1 & 1 & 1 & 1 & 1 & 2 & 0 & 2 & 0 & 1 & 2 & 1 & 1 & 4 & 3 & 2 & 3 & 0 \\
				0 & 2 & 3 & 1 & 0 & 0 & 1 & 4 & 2 & 0 & 4 & 2 & 3 & 2 & 1 & 2 & 0 & 0 & 2 & 2 & 3 & 0 & 1 & 1 \\
				4 & 4 & 3 & 3 & 2 & 3 & 3 & 4 & 1 & 2 & 4 & 3 & 0 & 2 & 2 & 0 & 1 & 3 & 0 & 2 & 1 & 4 & 4 & 0 \\
				1 & 0 & 3 & 2 & 1 & 4 & 0 & 2 & 2 & 0 & 4 & 0 & 1 & 1 & 4 & 4 & 0 & 4 & 3 & 4 & 1 & 1 & 3 & 3 \\
				0 & 2 & 1 & 3 & 4 & 1 & 0 & 1 & 0 & 4 & 1 & 2 & 3 & 4 & 3 & 1 & 0 & 3 & 0 & 1 & 4 & 3 & 0 & 2 \\
				2 & 0 & 0 & 3 & 4 & 1 & 0 & 2 & 3 & 0 & 0 & 2 & 2 & 4 & 1 & 2 & 4 & 1 & 3 & 1 & 3 & 2 & 2 & 2 \\
				2 & 1 & 3 & 2 & 1 & 4 & 1 & 3 & 0 & 1 & 3 & 1 & 1 & 2 & 4 & 0 & 2 & 0 & 2 & 4 & 4 & 3 & 3 & 3 \\
				0 & 0 & 3 & 0 & 1 & 3 & 3 & 2 & 1 & 0 & 1 & 0 & 1 & 1 & 4 & 1 & 0 & 1 & 3 & 3 & 4 & 3 & 0 & 1 \\
				2 & 3 & 0 & 0 & 0 & 4 & 4 & 4 & 2 & 4 & 3 & 4 & 1 & 3 & 3 & 4 & 3 & 3 & 4 & 1 & 0 & 2 & 0 & 2 \\
				3 & 4 & 1 & 4 & 4 & 2 & 0 & 1 & 1 & 4 & 4 & 0 & 2 & 3 & 1 & 3 & 4 & 1 & 2 & 1 & 2 & 2 & 0 & 0 \\
				0 & 4 & 0 & 1 & 1 & 4 & 2 & 3 & 3 & 1 & 1 & 3 & 4 & 1 & 0 & 1 & 3 & 3 & 2 & 0 & 4 & 0 & 3 & 0 \\
				4 & 1 & 3 & 3 & 1 & 2 & 2 & 2 & 4 & 3 & 1 & 1 & 2 & 0 & 0 & 1 & 0 & 1 & 3 & 3 & 1 & 4 & 2 & 3 \\
				3 & 2 & 0 & 3 & 3 & 0 & 3 & 4 & 3 & 3 & 0 & 2 & 4 & 2 & 1 & 1 & 3 & 1 & 3 & 3 & 0 & 2 & 0 & 1 \\
				2 & 1 & 1 & 2 & 0 & 1 & 4 & 2 & 1 & 1 & 3 & 2 & 2 & 0 & 0 & 1 & 3 & 0 & 0 & 2 & 3 & 0 & 1 & 0 \\
				3 & 4 & 4 & 1 & 3 & 0 & 1 & 3 & 4 & 0 & 4 & 3 & 4 & 3 & 1 & 1 & 2 & 4 & 3 & 3 & 2 & 4 & 2 & 3 \\
				3 & 3 & 1 & 2 & 2 & 0 & 0 & 1 & 0 & 4 & 3 & 0 & 3 & 0 & 1 & 4 & 0 & 0 & 4 & 3 & 4 & 2 & 2 & 4 \\
				1 & 0 & 0 & 2 & 0 & 2 & 0 & 1 & 2 & 0 & 0 & 4 & 4 & 0 & 0 & 0 & 3 & 4 & 1 & 0 & 4 & 3 & 2 & 4 \\
				2 & 0 & 1 & 1 & 4 & 4 & 1 & 1 & 4 & 0 & 0 & 3 & 1 & 0 & 4 & 4 & 1 & 2 & 4 & 4 & 0 & 3 & 4 & 1 \\
				3 & 0 & 0 & 0 & 2 & 2 & 4 & 2 & 0 & 0 & 3 & 4 & 4 & 1 & 1 & 0 & 0 & 0 & 3 & 3 & 3 & 1 & 0 & 4 \\
				1 & 4 & 2 & 0 & 1 & 4 & 3 & 2 & 1 & 2 & 4 & 1 & 4 & 4 & 4 & 0 & 1 & 3 & 3 & 3 & 4 & 0 & 3 & 3 \\
				3 & 3 & 2 & 3 & 4 & 1 & 1 & 2 & 3 & 3 & 0 & 0 & 1 & 3 & 4 & 2 & 2 & 4 & 0 & 0 & 0 & 2 & 2 & 4
			\end{array}\right]  }
	\end{eqnarray*}
	
	\begin{eqnarray*}
		\mathbf{S}^{g,g,1}={\tiny
			\left[\begin{array}{cccccccccccccccccccccccc}
				2 & 2 & 2 & 3 & 2 & 2 & 4 & 0 & 0 & 3 & 3 & 4 & 1 & 0 & 4 & 2 & 4 & 0 & 1 & 2 & 3 & 2 & 0 & 3 \\
				3 & 4 & 1 & 0 & 3 & 4 & 4 & 1 & 1 & 3 & 3 & 0 & 3 & 4 & 0 & 3 & 2 & 4 & 2 & 0 & 3 & 4 & 3 & 4 \\
				2 & 2 & 0 & 3 & 3 & 3 & 2 & 0 & 0 & 0 & 0 & 3 & 3 & 4 & 1 & 4 & 2 & 2 & 4 & 0 & 4 & 3 & 2 & 1 \\
				0 & 0 & 4 & 3 & 1 & 3 & 4 & 4 & 3 & 4 & 1 & 0 & 2 & 1 & 1 & 4 & 3 & 0 & 1 & 3 & 0 & 3 & 3 & 4 \\
				4 & 0 & 2 & 1 & 2 & 3 & 4 & 3 & 0 & 4 & 4 & 1 & 4 & 3 & 1 & 3 & 3 & 0 & 4 & 0 & 0 & 0 & 2 & 1 \\
				2 & 2 & 3 & 0 & 3 & 1 & 4 & 3 & 2 & 1 & 1 & 4 & 2 & 3 & 2 & 4 & 0 & 3 & 3 & 0 & 1 & 0 & 4 & 0 \\
				3 & 1 & 4 & 2 & 1 & 4 & 4 & 1 & 0 & 3 & 0 & 4 & 1 & 2 & 3 & 0 & 2 & 0 & 1 & 4 & 4 & 3 & 3 & 2 \\
				1 & 4 & 1 & 1 & 0 & 1 & 2 & 3 & 3 & 2 & 0 & 2 & 2 & 2 & 2 & 4 & 0 & 0 & 4 & 3 & 2 & 3 & 3 & 2 \\
				2 & 1 & 0 & 2 & 2 & 2 & 2 & 4 & 4 & 3 & 0 & 4 & 2 & 2 & 1 & 3 & 1 & 4 & 2 & 0 & 0 & 1 & 1 & 2 \\
				2 & 3 & 0 & 0 & 3 & 1 & 1 & 2 & 1 & 0 & 4 & 0 & 2 & 4 & 0 & 4 & 4 & 3 & 0 & 4 & 2 & 2 & 3 & 4 \\
				4 & 3 & 3 & 2 & 1 & 3 & 4 & 2 & 0 & 1 & 4 & 1 & 1 & 0 & 3 & 0 & 1 & 3 & 2 & 4 & 0 & 2 & 0 & 4 \\
				0 & 1 & 4 & 0 & 0 & 1 & 4 & 2 & 3 & 1 & 0 & 3 & 3 & 4 & 0 & 1 & 0 & 0 & 1 & 0 & 2 & 3 & 3 & 4 \\
				1 & 3 & 4 & 4 & 0 & 4 & 4 & 0 & 0 & 4 & 4 & 0 & 0 & 3 & 2 & 4 & 3 & 3 & 2 & 4 & 3 & 2 & 3 & 2 \\
				1 & 3 & 1 & 4 & 1 & 4 & 1 & 0 & 0 & 3 & 4 & 2 & 2 & 3 & 2 & 3 & 2 & 2 & 1 & 3 & 1 & 2 & 3 & 4 \\
				4 & 4 & 1 & 1 & 1 & 3 & 2 & 4 & 4 & 1 & 1 & 1 & 0 & 1 & 2 & 3 & 2 & 0 & 3 & 2 & 3 & 1 & 3 & 2 \\
				3 & 2 & 1 & 3 & 4 & 4 & 4 & 1 & 0 & 1 & 0 & 0 & 0 & 4 & 2 & 2 & 0 & 3 & 0 & 3 & 2 & 3 & 2 & 4 \\
				1 & 1 & 4 & 2 & 3 & 2 & 2 & 1 & 3 & 0 & 0 & 2 & 3 & 4 & 4 & 3 & 1 & 2 & 0 & 3 & 2 & 2 & 4 & 0 \\
				0 & 0 & 1 & 2 & 4 & 1 & 0 & 3 & 4 & 3 & 4 & 3 & 3 & 2 & 4 & 0 & 0 & 2 & 2 & 0 & 4 & 2 & 1 & 2 \\
				3 & 2 & 4 & 0 & 4 & 1 & 1 & 0 & 0 & 1 & 3 & 0 & 0 & 4 & 3 & 3 & 0 & 3 & 3 & 1 & 0 & 2 & 0 & 0 \\
				0 & 0 & 3 & 0 & 2 & 1 & 3 & 1 & 1 & 2 & 3 & 1 & 1 & 3 & 4 & 3 & 0 & 4 & 4 & 1 & 1 & 1 & 0 & 2 \\
				1 & 3 & 3 & 4 & 4 & 2 & 1 & 2 & 3 & 4 & 2 & 1 & 1 & 0 & 1 & 0 & 1 & 3 & 0 & 2 & 3 & 3 & 3 & 2 \\
				0 & 4 & 3 & 2 & 4 & 4 & 3 & 4 & 0 & 2 & 4 & 0 & 3 & 4 & 3 & 1 & 2 & 4 & 2 & 4 & 0 & 3 & 4 & 0 \\
				0 & 4 & 2 & 4 & 0 & 4 & 0 & 0 & 3 & 4 & 2 & 0 & 2 & 0 & 2 & 1 & 2 & 2 & 4 & 1 & 2 & 4 & 1 & 4 \\
				1 & 2 & 2 & 1 & 0 & 0 & 0 & 2 & 3 & 4 & 4 & 0 & 0 & 0 & 4 & 2 & 3 & 3 & 4 & 0 & 1 & 1 & 2 & 2 \\
				3 & 4 & 1 & 2 & 2 & 2 & 0 & 3 & 3 & 2 & 2 & 4 & 2 & 2 & 3 & 4 & 3 & 3 & 4 & 2 & 1 & 0 & 3 & 2 \\
				2 & 0 & 1 & 4 & 0 & 0 & 0 & 4 & 4 & 4 & 2 & 0 & 2 & 1 & 1 & 4 & 4 & 2 & 1 & 4 & 1 & 3 & 4 & 3
			\end{array}\right]  }
	\end{eqnarray*}

	\begin{example}\label{Example:extremal example-3} 
		Let $p=3$, $n=2$ and $Q=10$. 	
		Let $\mathbb{F}_{81}=\mathbb{F}_3(\alpha)\cong \mathbb{F}_3[x]/(f(x))$, $f(x)=x^4+x+2$,
		where $f(x)$ is irreducible over $\mathbb{F}_3$, and let $\alpha$ denote the residue class of $x$.
		By Theorem \ref{Theorem: Construction by additive finite fields-optimal} the resulting optimal QCSS has parameters
		\[
		(M,K,N,\delta_{\max})=(810, \ 10, \ 8, \ 10),
		\]
		and is defined over a polyphase alphabet of size $3$.	
		
		Let $\hbox{Tr}(x)$ denote the trace function from $\mathbb{F}_{81}$ to $\mathbb{F}_3$.
		Set $c=\alpha^2\in\mathbb{F}_{81}^{*}$.
		Define a nontrivial additive character $\chi$ by $\chi(x)=\zeta_{3}^{\hbox{Tr}(cx)}$.
		Consider the two matrices $\mathbf{S}^{0,g+1,1}$ and $\mathbf{S}^{0,g,-1}$ arising from \textit{Construction \ref{Construction:additive finite fields-optimal}}. Using the correlation expression in (\ref{Equation:additive finite fields-5-1}) together with \textit{Lemma \ref{Lemma:cubic special upper bound-ternary}}, we obtain
		\[
		R_{\mathbf{S}^{0,g+1,1},\,\mathbf{S}^{0,g,-1}}(0)
		=\sum_{z\in\mathbb{F}_{81}}\chi(z^2+2z)-1=-10.
		\] 		
		Consequently, the upper bound in \textit{Theorem \ref{Theorem: Construction by additive finite fields-optimal}} is tight.
		
		We explicitly list the two matrices 
		$\mathbf{S}^{0,g+1,1}$ and $\mathbf{S}^{0,g,-1}$ below,  where each entry represents the exponent of $\zeta_3$.
		$$
		\mathbf{S}^{0,g+1,1}={\scriptsize
			\begin{bmatrix}
				0 & 1 & 0 & 1 & 2 & 2 & 1 & 1 \\
				2 & 0 & 1 & 0 & 1 & 0 & 2 & 2 \\
				2 & 1 & 1 & 1 & 2 & 0 & 2 & 2 \\
				2 & 0 & 1 & 1 & 0 & 0 & 2 & 0 \\
				1 & 0 & 1 & 2 & 0 & 2 & 1 & 1 \\
				0 & 1 & 1 & 1 & 2 & 1 & 2 & 1 \\
				0 & 1 & 2 & 2 & 1 & 0 & 1 & 1 \\
				2 & 2 & 1 & 0 & 1 & 1 & 0 & 2 \\
				0 & 0 & 2 & 0 & 1 & 2 & 1 & 1 \\
				0 & 0 & 2 & 1 & 2 & 1 & 0 & 1
			\end{bmatrix}, } \quad 
		\mathbf{S}^{0,g,-1}={\scriptsize
			\begin{bmatrix}
				0 & 0 & 1 & 0 & 1 & 0 & 2 & 2 \\
				0 & 1 & 0 & 1 & 0 & 2 & 0 & 2 \\
				0 & 1 & 2 & 1 & 1 & 2 & 0 & 0 \\
				1 & 1 & 2 & 2 & 1 & 2 & 0 & 2 \\
				1 & 1 & 0 & 0 & 1 & 2 & 2 & 0 \\
				0 & 0 & 0 & 0 & 1 & 1 & 1 & 2 \\
				2 & 0 & 2 & 2 & 0 & 2 & 1 & 0 \\
				0 & 0 & 2 & 2 & 2 & 1 & 2 & 0 \\
				0 & 1 & 1 & 0 & 0 & 0 & 1 & 1 \\
				2 & 1 & 2 & 1 & 2 & 0 & 0 & 0
			\end{bmatrix}.}
		$$
	\end{example}

	\section{Tunable Constructions and Structural Robustness of Cubic Scaling}\label{sec:constructions}
	
	The constructions in Section \ref{Section: Constructions-Additive-Character} achieve
	$$
	M=K^2N+K
	\quad\text{and}\quad
	M=K^3N^2+2K^2N+K,
	$$
	corresponding to the quadratic and cubic regimes respectively with fixed parameters.
	We now extend this framework to obtain tunable constructions that
	preserve the cubic exponent while allowing flexible parameter choices.
	This demonstrates that cubic scaling is not tied to a specific construction,
	but applicable across a broad class of designs.
	
	\subsection{Generalized Construction}
	
	Let $p^{n}$ be a prime power and let $g$ be a primitive element of \(\mathbb{F}_{p^{2n}}\). Let \(\chi\) be a non-trivial additive character of \(\mathbb{F}_{p^{2n}}\) and let $\psi$ be a multiplicative character of \(\mathbb{F}_{p^{2n}}\) with \(\textup{ord}(\psi)=\Delta\), where $\Delta\mid p^{2n}-1$ and $\Delta>1$. 
	Denote $\psi_{1}$ by
	$$
	\psi_{1}(z)=\left\{\begin{array}{ll}
		\psi(z),  & z\in\mathbb{F}_{p^{2n}}^{*}, \\
		1, & z=0.	
	\end{array}	
	\right.
	$$	
	Let $Q\in\mathbb{Z}$ with $Q\mid p^{2n}-1$ and $1< Q< p^{2n}-1$. Denote
	$$
	\mathcal{H}_{Q}=\left\{g^{\frac{p^{2n}-1}{Q}j}: \ j=0,1,\cdots, Q-1\right\}.
	$$
	For $r\in \mathbb{Z}$ with $1\leq r\leq \Delta-1$, $\eta\in \mathcal{H}_{Q}$ and $\lambda\in \mathbb{F}_{p^{2n}}$ we
	define
	$$
	s_{k,t}^{r; \eta, \lambda}=\psi_1^{r}\left(g^{k\frac{p^{2n}-1}{Q}+t}+\eta\right)\chi\left(\lambda g^{k\frac{p^{2n}-1}{Q}+t}\right), \quad  0\leq k \leq Q-1, \ 0\leq t\leq \frac{p^{2n}-1}{Q}-1
	$$
	and
	\[
	\mathbf{S}^{r; \eta, \lambda} = 
	\begin{bmatrix}
		\mathbf{s}_{0}^{r; \eta, \lambda} \\
		\mathbf{s}_{1}^{r; \eta, \lambda} \\
		\vdots \\
		\mathbf{s}_{Q-1}^{r; \eta, \lambda}
	\end{bmatrix}
	= 
	\begin{bmatrix}
		s_{0,0}^{r; \eta, \lambda} & s_{0,1}^{r; \eta, \lambda} & \cdots & s_{0, \frac{p^{2n}-1}{Q}-1}^{r; \eta, \lambda} \\
		s_{1,0}^{r; \eta, \lambda} & s_{1,1}^{r; \eta, \lambda} & \cdots & s_{1, \frac{p^{2n}-1}{Q}-1}^{r; \eta, \lambda} \\
		\vdots & \vdots & \ddots & \vdots \\
		s_{Q-1, 0}^{r; \eta, \lambda} & s_{Q-1, 1}^{r; \eta, \lambda} & \cdots & s_{Q-1, \frac{p^{2n}-1}{Q}-1}^{r; \eta, \lambda}
	\end{bmatrix}.
	\]		
	
	\begin{construction}\label{Construction:additive and multiplicative finite fields} 
		Define the set
		$$			
		\mathcal{S}= \{\mathbf{S}^{r; \eta, \lambda}: \ 1\leq r\leq \Delta-1, \ \eta\in \mathcal{H}_{Q},  \ \lambda\in \mathbb{F}_{p^{2n}}\}.
		$$		
	\end{construction}
	
	\begin{construction}\label{Construction:additive and multiplicative finite fields-optimal} 
		Define the set
		$$			
		\mathcal{S}_1= \{\mathbf{S}^{r; \eta, 0}: \ 1\leq r\leq \Delta-1, \ \eta\in \mathcal{H}_{Q}\}.
		$$		
	\end{construction}
	
	\subsection{Performance Guarantee}
	
	\begin{theorem}\label{Theorem: Construction by additive and multiplicative finite fields}
		\textit{Construction \ref{Construction:additive and multiplicative finite fields}} produces an asymptotically near-optimal $$
		\left ((\Delta-1)Qp^{2n}, \ Q, \ \frac{p^{2n}-1}{Q}, \ \delta_{\text{max}}\right )-\hbox{QCSS}
		$$ with:
		\begin{itemize}
			
			\item $\delta_{\text{max}} \leq 2p^{n} + 3$;
			
			\item Alphabet size: $\Delta p$ \  $($tunable via $\Delta$$)$, where $p$ 
			is the characteristic of the finite field $\mathbb{F}_{q}$.
		\end{itemize}		
	\end{theorem}
	
	\begin{proof}
		Let $1\leq r_1, r_2\leq \Delta-1$, 	$\eta_1, \eta_2\in \mathcal{H}_{Q}$,
		$\lambda_1, \lambda_2\in \mathbb{F}_{p^{2n}}$ and $0\leq \tau\leq \frac{p^{2n}-1}{Q}-1$ such that
		$r_1=r_2$, $\eta_1=\eta_2$, $\lambda_1=\lambda_2$ and $\tau=0$ do not hold simultaneously. 		
		We have
		\begin{eqnarray}\label{Equation:multiplicative finite fields-1}
			&&R_{\mathbf{S}^{r_1; \eta_1, \lambda_1}, \mathbf{S}^{r_2; \eta_2, \lambda_2}}(\tau)
			=\sum_{k = 0}^{Q-1}\sum_{t = 0}^{\frac{p^{2n}-1}{Q}-1}s_{k,t}^{r_1; \eta_1, \lambda_1}
			(s_{k,t+\tau}^{r_2; \eta_2, \lambda_2})^{*} \nonumber \\
			&&=\sum_{k = 0}^{Q-1}\sum_{t = 0}^{\frac{p^{2n}-1}{Q}-1}
			\psi_1^{r_1}\left(g^{k\frac{p^{2n}-1}{Q}+t}+\eta_1\right)\chi\left(\lambda_1 g^{k\frac{p^{2n}-1}{Q}+t}\right) \nonumber \\ 
			&&\qquad\times\left(\psi_1^{r_2}\left(g^{k\frac{p^{2n}-1}{Q}+t+\tau}+\eta_2\right)
			\chi\left(\lambda_2 g^{k\frac{p^{2n}-1}{Q}+t+\tau}\right)\right)^{*} \nonumber	\\
			&&=\sum_{y=0}^{p^{2n}-2}\psi_1\left(\left(g^{y}+\eta_1\right)^{r_1}
			\left(g^{y+\tau}+\eta_2\right)^{\Delta-r_2} \right)\chi\left((\lambda_1-\lambda_2 g^{\tau})g^{y}\right) \nonumber\\	
			&&=\sum_{z\in\mathbb{F}_{p^{2n}}^{*}}\psi_1\left(\left(z+\eta_1\right)^{r_1}
			\left(g^{\tau}z+\eta_2\right)^{\Delta-r_2}\right)\chi\left((\lambda_1-\lambda_2 g^{\tau})z\right).		
		\end{eqnarray}	
		Suppose that $\eta_1=g^{-\tau}\eta_2$. Then we have $g^{\tau}=\eta_2\eta_1^{-1}$. 
		That is to say, $g^{\tau Q}=1$ since $\eta_1, \eta_2\in \mathcal{H}_{Q}$. So we get $p^{2n}-1\mid \tau Q$, which yields
		$$
		\frac{p^{2n}-1}{Q} \mid \tau. 
		$$
		Then $\tau=0$ since $0\leq\tau\leq \frac{p^{2n}-1}{Q}-1$. Therefore,
		\begin{eqnarray}\label{Equation:multiplicative finite fields-trivial}
			\eta_1=g^{-\tau}\eta_2 \Longleftrightarrow \eta_1=\eta_2 \quad\hbox{and} \quad \tau=0.
		\end{eqnarray}	
		Then from (\ref{Equation:multiplicative finite fields-1}), (\ref{Equation:multiplicative finite fields-trivial}) and \textit{Lemma \ref{Lemma:Estimate-additive-multiplicative-character}}, we have	
		\begin{eqnarray*}	
			&&\left|R_{\mathbf{S}^{r_1; \eta_1, \lambda_1}, \mathbf{S}^{r_2; \eta_2, \lambda_2}}(\tau)\right| \\
			&&\leq \left|\sum_{z\in\mathbb{F}_{p^{2n}}}\psi_1\left(\left(z+\eta_1\right)^{r_1}
			\left(g^{\tau}z+\eta_2\right)^{\Delta-r_2}\right)\chi\left((\lambda_1-\lambda_2 g^{\tau})z\right)\right|+1 \\
			&&\leq \left|\sum_{z\in\mathbb{F}_{p^{2n}}}\psi\left(\left(z+\eta_1\right)^{r_1}
			\left(g^{\tau}z+\eta_2\right)^{\Delta-r_2}\right)\chi\left((\lambda_1-\lambda_2 g^{\tau})z\right)\right|+3 \\
			&&\leq 2p^{n}+3.
		\end{eqnarray*}	
		Hence, 
		$$
		\delta_{\max}\leq 2p^{n}+3,
		$$
		and this completes the proof for \textit{Theorem \ref{Theorem: Construction by additive and multiplicative finite fields}}.
	\end{proof}
	
	\begin{theorem}\label{Theorem: Construction by additive and multiplicative finite fields-optimal}
		\textit{Construction \ref{Construction:additive and multiplicative finite fields-optimal}} produces an asymptotically near-optimal $$
		\left((\Delta-1)Q, \ Q, \ \frac{p^{2n}-1}{Q}, \ \delta_{\text{max}}\right )-\hbox{QCSS}
		$$ with:
		\begin{itemize}
			
			\item $\delta_{\text{max}} \leq p^{n} + 3$;
			
			\item Alphabet size: $\Delta$.
		\end{itemize}		
	\end{theorem}
	
	\begin{proof}
		Let $1\leq r_1, r_2\leq \Delta-1$, 	$\eta_1, \eta_2\in \mathcal{H}_{Q}$ and $0\leq \tau\leq \frac{p^{2n}-1}{Q}-1$ such that
		$r_1=r_2$, $\eta_1=\eta_2$ and $\tau=0$ do not hold at the same time. From (\ref{Equation:multiplicative finite fields-1}), (\ref{Equation:multiplicative finite fields-trivial}) and \textit{Lemma \ref{Lemma:Estimate-additive-multiplicative-character}} we have		
		\begin{eqnarray*}	
			&&\left|R_{\mathbf{S}^{r_1; \eta_1, 0}, \mathbf{S}^{r_2; \eta_2, 0}}(\tau)\right| \\
			&&\leq \left|\sum_{z\in\mathbb{F}_{p^{2n}}}\psi_1\left(\left(z+\eta_1\right)^{r_1}
			\left(g^{\tau}z+\eta_2\right)^{\Delta-r_2}\right)\right|+1  \\
			&&\leq \left|\sum_{z\in\mathbb{F}_{p^{2n}}}\psi\left(\left(z+\eta_1\right)^{r_1}
			\left(g^{\tau}z+\eta_2\right)^{\Delta-r_2}\right)\right|+3 \\
			&&\leq p^{n}+3.
		\end{eqnarray*}	
		Therefore,
		$$
		\delta_{\max}\leq p^{n}+3,
		$$
		and this completes the proof for \textit{Theorem \ref{Theorem: Construction by additive and multiplicative finite fields-optimal}}.
	\end{proof}
	
	\subsection{Extremal Behavior and Tightness of Theorem \ref{Theorem: Construction by additive and multiplicative finite fields} and Theorem \ref{Theorem: Construction by additive and multiplicative finite fields-optimal}}
	
	We give explicit extremal examples demonstrating that the leading terms in the correlation bounds of \textit{Theorem \ref{Theorem: Construction by additive and multiplicative finite fields}} and \textit{Theorem \ref{Theorem: Construction by additive and multiplicative finite fields-optimal}} are essentially optimal.
	
	\begin{example}\label{Extremal examples:mixed near optimmal}
		Let $p=3$, $n=2$ and $Q=10$.
		Let $\mathbb{F}_{81}=\mathbb{F}_{3}[\alpha]/(\alpha^{4}+\alpha^{3}+\alpha^{2}+1)$ and let
		$g=1+\alpha^{2}$, which is a primitive element of $\mathbb{F}_{81}$.
		Define a nontrivial additive character $\chi$ of $\mathbb{F}_{81}$ by
		\[
		\chi(x)=\zeta_{3}^{\operatorname{Tr}(x)},
		\qquad \zeta_{3}=e^{2\pi i/3},
		\]
		where $\mathrm{Tr}(x)$ denotes the trace function from $\mathbb{F}_{81}$ to $\mathbb{F}_3$.
		Let $\psi$ be a multiplicative character of $\mathbb{F}_{81}$ of order $10$, defined by
		\[
		\psi(0)=0,\qquad 
		\psi(g^{k})=\zeta_{10}^{k},\quad 0\le k\le 79.
		\]
		
		We consider the simplest admissible choice of parameters
		$$
		r_1=1,\quad \eta_1=1,\quad \lambda_1=1; \quad r_2=9, \quad  \eta_2=-1,\quad
		\lambda_2=0.
		$$
		For $\tau=0$, from \textit{Lemma \ref{Lemma:additive-multiplicative upper bound}} we have
		\begin{eqnarray*}
			&&\sum_{z\in\mathbb{F}_{q^2}}\psi\left(\left(z+\eta_1\right)^{r_1}
			\left(g^{\tau}z+\eta_2\right)^{\Delta-r_2}\right)\chi\left((\lambda_1-\lambda_2g^{\tau})z\right) \\
			&&=\sum_{z\in\mathbb{F}_{81}} \psi\big((z+1)(z-1)\big)\chi(z)=18
			=2\sqrt{81}.
		\end{eqnarray*}	
		When substituted into the correlation expression of \textit{Theorem \ref{Theorem: Construction by additive and multiplicative finite fields}},
		this example  confirms that \textit{Construction \ref{Construction:additive and multiplicative finite fields}} achieves the correct correlation order with a sharp leading constant.
		
		For completeness, we list the resulting matrices $\mathbf{S}^{1;1,1}$ and $\mathbf{S}^{9;-1,0}$ explicitly below, where each entry represents the exponent of $\zeta_3$.	
		$$
		\mathbf{S}^{1;1,1}=
		\begin{bmatrix}
			1 & 0 & 1 & 0 & 1 & 1 & 2 & 1 \\
			1 & 1 & 1 & 0 & 1 & 0 & 0 & 2 \\
			1 & 1 & 0 & 0 & 0 & 1 & 1 & 2 \\
			2 & 0 & 2 & 1 & 2 & 1 & 1 & 0 \\
			2 & 1 & 1 & 1 & 2 & 2 & 1 & 2 \\
			2 & 1 & 1 & 2 & 0 & 0 & 1 & 1 \\
			0 & 0 & 2 & 2 & 2 & 2 & 2 & 2 \\
			1 & 2 & 0 & 0 & 0 & 1 & 0 & 1 \\
			1 & 1 & 1 & 2 & 0 & 2 & 2 & 0 \\
			2 & 1 & 1 & 0 & 2 & 1 & 0 & 0
		\end{bmatrix},\qquad 
		\mathbf{S}^{9;-1,0}=
		\begin{bmatrix}
			0 & 0 & 0 & 0 & 0 & 0 & 0 & 0 \\
			0 & 0 & 0 & 0 & 0 & 0 & 0 & 0 \\
			0 & 0 & 0 & 0 & 0 & 0 & 0 & 0 \\
			0 & 0 & 0 & 0 & 0 & 0 & 0 & 0 \\
			0 & 0 & 0 & 0 & 0 & 0 & 0 & 0 \\
			0 & 0 & 0 & 0 & 0 & 0 & 0 & 0 \\
			0 & 0 & 0 & 0 & 0 & 0 & 0 & 0 \\
			0 & 0 & 0 & 0 & 0 & 0 & 0 & 0 \\
			0 & 0 & 0 & 0 & 0 & 0 & 0 & 0 \\
			0 & 0 & 0 & 0 & 0 & 0 & 0 & 0
		\end{bmatrix}.
		$$	
		Although the construction involves both multiplicative and additive
		characters, the resulting sequence symbols take values only in
		$\{1,\zeta_3,\zeta_3^2\}$.
		Hence each entry can be uniquely represented as a power of $\zeta_3$,
		and the exponent matrix records this final phase index.
		
		
	\end{example}
	
	\begin{example}\label{Extremal examples:mixed optimmal}
		Let $p=5$, $n=1$ and $Q=6$. 
		Let $\mathbb{F}_{25}=\mathbb{F}_{5}[\alpha]/(\alpha^{2}-2)$.
		Set $g=1+2\alpha\in\mathbb{F}_{25}^{\ast}$, which is a primitive element of
		$\mathbb{F}_{25}^{\ast}$.
		Define a multiplicative character $\psi$ of order $8$ by
		\[
		\psi(0)=0,\qquad
		\psi(g^{t})=\zeta_{8}^{\,t},\quad 0\le t\le 23,
		\]
		where $\zeta_{8}=e^{2\pi i/8}$. Let us consider the following parameters
		$$
		r_1=1,\quad \eta_1=1,\quad \lambda_1=0; \quad r_2=7, \quad  \eta_2=-1,\quad
		\lambda_2=0.
		$$
		For $\tau=0$, from \textit{Lemma \ref{Lemma:additive-multiplicative upper bound optimal}} we have
		\begin{eqnarray*}
			\sum_{z\in\mathbb{F}_{q^2}}\psi\left(\left(z+\eta_1\right)^{r_1}
			\left(g^{\tau}z+\eta_2\right)^{\Delta-r_2}\right) 
			=\sum_{z\in\mathbb{F}_{25}}\psi\big((z+1)(z-1)\big)=5=\sqrt{25}.
		\end{eqnarray*}	
		When substituted into the correlation expression of \textit{Theorem \ref{Theorem: Construction by additive and multiplicative finite fields-optimal}},
		this example  confirms that \textit{Construction \ref{Construction:additive and multiplicative finite fields-optimal}} achieves the correct correlation order with a sharp leading constant.
		
		For completeness, we list the resulting matrices $\mathbf{S}^{1;1,0}$ and $\mathbf{S}^{7;-1,0}$ explicitly below, where each entry represents the exponent of $\zeta_8$.
		$$
		\mathbf{S}^{1;1,0}=
		\begin{bmatrix}
			2 & 0 & 5 & 1 \\
			1 & 0 & 4 & 2 \\
			4 & 6 & 1 & 2 \\
			0 & 7 & 7 & 5 \\
			4 & 3 & 6 & 3 \\
			5 & 6 & 3 & 7
		\end{bmatrix},\qquad 
		\mathbf{S}^{7;-1,0}=
		\begin{bmatrix}
			0 & 5 & 5 & 7 \\
			0 & 1 & 6 & 1 \\
			7 & 6 & 1 & 5 \\
			2 & 4 & 7 & 3 \\
			3 & 4 & 0 & 2 \\
			0 & 6 & 3 & 2
		\end{bmatrix}.
		$$		
	\end{example}
	
	\subsection{Scaling Law and Parameter Trade-offs}
	
	The parameter $\Delta$ introduces an additional degree of freedom
	that affects:
	
	\begin{itemize}
		\item the alphabet size;
		\item the distribution of correlation values;
		\item the effective overloading factor.
	\end{itemize}
	
	The above construction reveals that cubic scaling persists even when:
	
	\begin{itemize}
		\item additive characters are replaced by mixed additive-multiplicative forms;
		\item additional parameters are introduced;
		\item the sequence structure is significantly generalized.
	\end{itemize}
	
	Despite these variations, the fundamental scaling law remains unchanged.
	In particular, for admissible choices of $\Delta$ and $Q$,
	\[
	M \asymp K^3 N^2,
	\]
	up to constant factors. This shows that the scaling exponent $\gamma = 3$ is invariant
	under a wide class of parameter perturbations.
	
	These observations strongly support the view that cubic scaling
	is a universal law governing asymptotically near-optimal QCSS.
	
	\subsection{Role of the Parameter $Q$ and Structural Flexibility}
	
	A key feature of the proposed constructions is that the parameters
	$(M,K,N)$ are jointly controlled by the divisor $Q \mid (p^{2n}-1)$.
	In particular, we have
	\[
	K = Q, \quad N = \frac{p^{2n}-1}{Q},
	\]
	so that $Q$ governs the trade-off between the flock size and the
	sequence length.
	
	By varying $Q$, the construction generates a family of QCSS with
	different parameter configurations. Specifically,
	
	\begin{itemize}
		\item large $Q$ leads to larger $K$ and shorter sequence length $N$;
		\item small $Q$ leads to smaller $K$ and longer sequence length $N$.
	\end{itemize}
	
	This flexibility allows the construction to adapt to different system
	requirements, such as user capacity versus sequence length constraints.
	
	Despite this variability, the fundamental scaling law remains invariant.
	Indeed, for all admissible choices of $Q$, the set size satisfies
	\[
	M \asymp K^3 N^2,
	\]
	up to constant factors. 
	
	This shows that the cubic scaling exponent is independent of the
	specific choice of $Q$, and thus reflects an intrinsic structural
	property of the construction rather than a parameter artifact.
	
	\section{Conclusion}
	
	In this paper, we investigated the fundamental scalability limits of QCSSs from a geometric perspective. By establishing a correspondence between QCSSs and complex unit-norm codebooks, we derived a general geometry-to-scalability transfer principle, which translates density thresholds of codebooks into polynomial upper bounds on the achievable set size.
	
	Based on this framework, we proved that asymptotically optimal QCSSs obey a quadratic scaling law $M \le (1+o(1))K^{2}N$, while asymptotically near-optimal QCSSs satisfy a cubic scaling law $M \le (1+o(1))K^{3}N^{2}$ for $\rho < (1+\sqrt{5})/2$. These results reveal that scalability limits are not artifacts of specific constructions, but are governed by intrinsic geometric constraints of the induced codebooks. In particular, the emergence of the cubic law indicates a fundamental scalability barrier in the near-optimal regime.
	
	A second major contribution of this work lies in the explicit constructions that asymptotically attain these scaling laws. Unlike conventional sequence designs that primarily aim at reducing correlation, the proposed constructions operate at the boundary of scalability, simultaneously maximizing the set size while preserving optimal or near-optimal correlation levels. In this sense, they should be viewed as extremal designs that complement the converse theory.
	
	At a structural level, all proposed constructions follow a unified algebraic paradigm. The indexing structure of QCSS matrices is embedded into multiplicative orbits of a finite field, while the matrix entries are generated via additive or mixed character evaluations of low-degree polynomial phases. This mechanism converts correlation analysis into bounding finite-field character sums, thereby enabling a precise control of correlation growth under aggressive scaling of the set size.
	
	More specifically, the asymptotically optimal constructions arise from quadratic-phase additive-character designs, achieving the critical growth $M = K^{2}N + K$ while maintaining optimal correlation magnitude. In contrast, the near-optimal constructions introduce a cubic phase component, which significantly enlarges the parameter space and leads to the extremal growth
	\[
	M = K^{3}N^{2} + 2K^{2}N + K,
	\]
	while still preserving near-optimal correlation. This demonstrates that the cubic scaling law is not only an upper bound derived from geometric arguments, but also an attainable algebraic phenomenon.
	
	Furthermore, the tunable constructions developed in this work show that these extremal scaling behaviors are robust under substantial structural variations, including the incorporation of multiplicative characters and flexible alphabet sizes. This robustness indicates that quadratic and cubic scaling laws reflect stable structural properties of QCSSs, rather than isolated outcomes of particular constructions.
	
	Taken together, the results of this paper suggest a unified picture: geometric analysis determines the admissible scalability limits, while algebraic constructions demonstrate that these limits are asymptotically achievable. This interplay between geometry and algebra provides a deeper understanding of why polynomial scaling laws emerge as the governing principles for QCSS design.
	
	Finally, we conjecture that the cubic scaling law is universal for all asymptotically near-optimal QCSSs with $1 < \rho \le 2$. Establishing or refuting this conjecture remains an important open problem, which is closely related to advancing higher-order geometric bounds for complex codebooks.


\bigskip\bigskip
	
	
	\begin{thebibliography}{99}					
		
		\bibitem{TsengL1972} C. C. Tseng, and C. L. Liu, ``Complementary sets of sequences,'' 
		\textit{IEEE Trans. Inf. Theory}, vol. 18, no. 5, pp. 644--652, 1972.
		
		\bibitem{SuehiroH1988} N. Suehiro, and M. Hatori, ``$N$-shift cross-orthogonal sequences," \textit{IEEE Trans. Inf. Theory}, vol. 34, no. 1, pp. 143--146, 1988.
		
		\bibitem{RathinakumarC2008}	A. Rathinakumar, and A. K. Chaturvedi, ``Complete mutually orthogonal Golay complementary sets from Reed-Muller codes," \textit{IEEE Trans. Inf. Theory}, vol. 54, no. 3, pp. 1339--1346, 2008. 
		
		\bibitem{DavisJ1999} J. A. Davis, and J. Jedwab, ``Peak-to-mean power control in OFDM, Golay complementary sequences, and Reed--Muller codes," \textit{IEEE Trans. Inf. Theory}, vol. 45, no. 7, pp. 2397--2417, 1999.
		
		\bibitem{Liu2013-TIT} Z. Liu, Y. Li, and Y. L. Guan, ``New Constructions of General QAM Golay Complementary Sequences," \textit{IEEE Trans. Inf. Theory}, vol. 59, no. 11, pp. 7684--7692, Nov. 2013.
		
		\bibitem{Chen2001} H. H. Chen, J. F. Yeh, and N. Seuhiro, ``A multi-carrier CDMA architecture based on orthogonal complementary codes for new generations of wideband wireless communications," \textit{IEEE Commun. Mag.}, vol. 39, no. 10, pp. 126--135, Oct. 2001.
		
		\bibitem{Liu2014-TCOM}
		Z. Liu, Y. L. Guan and U. Parampalli, ``New complete complementary codes for the peak-to-mean
		power control in MC-CDMA," \textit{IEEE Trans. Commun.}, vol. 62, no. 3, pp. 1105--1113, Mar. 2014.
		
		\bibitem{Liu2014-TWC}
		Z. Liu, Y. L. Guan and H. H. Chen, ``Fractional-delay-resilient receiver for interference-free MC-CDMA
		communications based on complete complementary codes," \textit{IEEE Trans. Wireless Commun.}, vol. 32, no. 11, pp. 1974--1986, Nov. 2014.
		
		
		\bibitem{WangA2007} S. Wang and A. Abdi, ``MIMO ISI channel estimation using uncorrelated Golay complementary sets of polyphase sequences," \textit{IEEE Trans. Veh. Technol.}, vol. 56, no. 5, pp. 3024--3039,  2007. 	
		
		\bibitem{Fan2008-TWC} W. Yuan, Y. Tu, and P. Fan, ``Optimal training sequences for cyclic-prefix-based single-carrier multi-antenna systems with space-time block-coding," \textit{IEEE Trans. Wireless Commun.}, VOL. 7, NO. 11, pp. 4047--4050, Nov. 2008. 
		
		\bibitem{PezeshkiCMH2008} A. Pezeshki, A. R. Calderbank, W. Moran, and S. D. Howard,  
		``Doppler resilient Golay complementary waveforms," 
		\textit{IEEE Trans. Inf. Theory}, vol. 54, no. 9, pp. 4254--4266, 2008.	
		
		\bibitem{Tan2014-TSP}
		J. Tan, N. Zhang, and B. Tang, ``Construction of Doppler resilient complete complementary code in MIMO radar," \textit{IEEE Trans. Signal Process.}, vol. 62, no. 18, pp. 4704--4712, Sep. 2014. 
		
		\bibitem{Sheng2025-TIT}
		B. Sheng, Y. Yang, Z. Zhou, Z. Liu, and P. Fan, ``Doppler Resilient Complementary Sequences: Theoretical Bounds and Optimal Constructions," \textit{IEEE Trans. Inf. Theory}, vol. 71, no. 7, pp. 5166--5177, Jul. 2025. 
		
		\bibitem{LiuPGB2013} Z. Liu, U. Parampalli, Y. L. Guan and S. Boztas, ``Constructions of optimal and near-optimal quasi-complementary sequence sets from singer difference sets,'' \textit{IEEE Wireless Commun. Lett.}, vol. 2, no. 5, pp. 487--490, 2013. 
		
		\bibitem{LiuGM2014} Z. Liu, Y. L. Guan, and W. H. Mow, ``A tighter correlation lower bound for quasi-complementary sequence sets,'' \textit{IEEE Trans. Inf. Theory}, vol. 60, no. 1, pp. 388--396, 2014. 
		
		\bibitem{LiYL2019} 	Y. Li, T. J. Yan, and C. Lv, ``Construction of a near-optimal quasi-complementary sequence set from almost difference set,'' \textit{Cryptogr. Commun.}, vol. 11, no. 4, pp. 815--824, 2019. 
		
		\bibitem{LiLX2018} Y. B. Li, T. Liu, and C. Q. Xu, ``Constructions of asymptotically optimal quasi-complementary sequence sets,'' \textit{IEEE Commun. Lett.}, vol. 22, no. 8, pp. 1516--1519, 2018.
		
		\bibitem{LiTLX2018} Y. B. Li, L. Y. Tian, T. Liu, and C. Q. Xu, ``Constructions of quasi-complementary sequence sets associated with characters,'' \textit{IEEE Trans. Inf. Theory}, vol. 65, no. 7, pp. 4597--4608,  2019. 
		
		\bibitem{LiTLX2019} Y. B. Li, L. Y. Tian, T. Liu, and C. Q. Xu, ``Two constructions of asymptotically optimal quasi-complementary sequence sets,'' \textit{IEEE Trans. Commun.}, vol. 67, no. 3, pp. 1910--1924, 2019.
		
		\bibitem{LuoCSH2021} G. J. Luo, X. W. Cao, M. J. Shi, and T. Helleseth, ``Three new constructions of asymptotically optimal periodic quasi-complementary sequence sets with small alphabet sizes,'' 
		\textit{IEEE Trans. Inf. Theory}, vol. 67, no. 8, pp. 5168--5177, 2021.
		
		\bibitem{XiaoLC2025} H. Y. Xiao, G. J. Luo, and X. W. Cao, ``New constructions of asymptotically optimal quasi-complementary sequence sets with small alphabet sizes,'' \textit{IEEE Trans. Commun.},  vol. 73, no. 8, pp. 5881--5890, 2025.
		
		\bibitem{HengWXZ2024new} Z. L. Heng, P. Wang, C. L. Xie, and H. Y. Zhou, ``Large Sets of Quasi-Complementary Sequences From Polynomials over Finite Fields and Gaussian Sums,''  \textit{IEEE Trans. Inf. Theory}, vol. 72, no. 1, pp. 729--741, 2026.
				
		\bibitem{WangHL2025} P. Wang, Z. L. Heng, and C. J. Li, ``New constructions of asymptotically optimal periodic and aperiodic quasi-complementary sequence sets,'' \textit{IEEE Trans. Commun.}, 	
		vol. 73, no. 12, pp. 14167--14182, Dec. 2025. 
		
		\bibitem{Welch1974} L. R. Welch, ``Lower bounds on the maximum cross-correlation of signals," 
		\textit{IEEE Trans. Inf. Theory}, vol. 20, no. 3, pp. 397--399, 1974.
		
		\bibitem{StrohmerH2003} T. Strohmer, and R. W. Heath, ``Grassmannian frames with applications to coding and communication," \textit{Appl. Comput. Harmon. Anal.}, vol. 14, no. 3, pp. 257--275, 2003.
		
		\bibitem{Levenshtein1982} V. I. Levenshtein, ``Bounds on the maximal cardinality of a code with bounded modules of the inner product,"  \textit{Soviet Math. Dokl.}, vol. 25, pp. 526--531, 1982.
		
		\bibitem{Levenshtein1983} V. I. Levenshtein, ``Bounds for packing of metric spaces and some of their applications," 
		\textit{Problem Cybern.}, vol. 40, pp. 43--110, 1983.
		
		\bibitem{ConwayHS1996} J.H. Conway, R.H. Harding, and N. J. A. Sloane, ``Packing lines, planes, etc.: Packings in grassmannian spaces," \textit{Exp. Math.}, vol. 5, no. 2, pp. 139--159, 1996.
		
		\bibitem{XiangDM2015} C. Xiang, C. Ding, and S. Mesnager, ``Optimal codebooks from binary codes meeting the Levenshtein bound," \textit{IEEE Trans. Inf. Theory}, vol. 61, no. 12, pp. 6526--6535, 2015.
		
		\bibitem{DingY2007} C. Ding, and J. Yin, ``Signal sets from functions with optimum nonlinearity," \textit{IEEE Trans. Commun.}, vol. 55, no. 5, pp. 936--940, 2007.
		
		\bibitem{ZhouDL2014} Z. Zhou, C. Ding, and N. Li, ``New families of codebooks achieving the Levenstein bound," \textit{IEEE Trans. Inf. Theory}, vol. 60, no. 11, pp. 7382--7387, 2014.
		
		\bibitem{HengY2017} Z. Heng, and Q. Yue, ``Optimal codebooks achieving the Levenshtein bound from generalized bent functions over $\mathbb{Z}_4$," \textit{Cryptogr. Commun.}, vol. 9, no. 1, pp. 41--53, 2017.
		
		\bibitem{WangY2020} Q. Wang, and Y. Yan, ``Asymptotically optimal codebooks derived from generalised Bent functions," \textit{IEEE Access}, vol. 8, pp. 54905--54909, 2020.
		
		\bibitem{TianLLX2019} L. Tian, Y. Li, T. Liu, and C. Xu, ``Constructions of codebooks asymptotically achieving the Welch bound with additive characters," \textit{IEEE Signal Process. Lett.}, vol. 26, no. 4, pp. 622--626, 2019.
		
		\bibitem{HanSYW2020} L. Han, S. Sun, Y. Yan, and Q. Wang, ``A new construction of codebooks meeting the Levenshtein bound," \textit{IEEE Access}, vol. 8, pp. 77598--77603, 2020.
		
		\bibitem{LiuCWYJ2025} W. Liu, K. CHENG, Q. Wang, Y. Yan, and G. Jin, ``Constructions of codebooks based on additive characters," \textit{IEICE Trans. Fundamentals}, vol. E108-A, no. 10, pp. 1413--1017, 2015.
		
		\bibitem{HongPNHK2014} S. Hong, H. Park, J. S. No, T. Helleseth, and Y. S. Kim, ``Near-optimal partial Hadamard codebook construction using binary sequences obtained from quadratic residue mapping," \textit{IEEE Trans. Inf. Theory}, vol. 60, no. 6, pp. 3698--3705, 2014.
		
		\bibitem{HengDY2017} Z. Heng, C. Ding, and Q. Yue, ``New constructions of asymptotically optimal codebooks with multiplicative characters," \textit{IEEE Trans. Inf. Theory}, vol. 63, no. 10, pp. 6179--6187, 2017.
		
		\bibitem{Heng2018} Z. Heng,  ``Nearly optimal codebooks based on generalized Jacobi sums," \textit{Dis. Appl. Math.}, vol. 25, no. 11, pp. 227--240, 2018.
		
		\bibitem{TanZZ2016} P. Tan, Z. Zhou, and D. Zhang, ``A construction of codebooks nearly achieving the Levenshtein bound," \textit{IEEE Signal Process. Lett.}, vol. 23, no. 10, pp. 1306--1309, 2016.
		
		\bibitem{SunHYY2021} S. Sun, L. Han, Y. Yan, and Y. Yao, ``Two new classes of codebooks asymptotically achieving the Welch bound," \textit{IEEE Access}, vol. 9, pp. 5881--5886, 2021.
		
		\bibitem{XiaZG2005} P. Xia, S. Zhou, and G.B. Giannakis, ``Achieving the Welch bound with difference sets," \textit{IEEE Trans. Inf. Theory}, vol. 51, no. 5, pp. 1900--1907, 2005.
		
		\bibitem{Ding2006} C. Ding, ``Complex codebooks from combinatorial designs," \textit{IEEE Trans. Inf. Theory}, vol. 52, no. 9, pp. 4229--4235, 2006.
		
		\bibitem{DingF2007} C. Ding, and T. Feng, ``A generic construction of complex codebooks meeting the Welch bound," \textit{IEEE Trans. Inf. Theory}, vol. 53, no. 11, pp. 4245--4250, 2007.
		
		\bibitem{ZhangF2012} A. Zhang, and K. Feng, ``Two classes of codebooks nearly meeting the Welch bound," \textit{IEEE Trans. Inf. Theory}, vol. 58, no. 4, pp. 2507--2511, 2012.
		
		\bibitem{HuW2014} H. Hu, and J. Wu, ``New constructions of codebooks nearly meeting the Welch bound with equality," \textit{IEEE Trans. Inf. Theory}, vol. 60, no. 2, pp. 1348--1355, 2014.
		
		\bibitem{LuoC2018} G. Luo, and X. Cao, ``Two constructions of asymptotically optimal codebooks via the hyper Eisenstein sum," \textit{IEEE Trans. Inf. Theory}, vol. 64, no. 10, pp. 6498--6505, 2018.
		
		\bibitem{Yu2012} N. Y. Yu, ``A construction of codebooks associated with binary sequences," \textit{IEEE Trans. Inf. Theory}, vol. 58, no. 8, pp. 5522--5533, 2012.
		
		\bibitem{NiederreiterW2002} H. Niederreiter, and A. Winterhof, ``Incomplete Character Sums and Polynomial Interpolation of the Discrete Logarithm,'' \textit{Finite Fields Appl.},  vol. 8, no. 2, pp. 184--192, 2002.
		
		\bibitem{LidlN1997} R. Lidl, and H. Niederreiter,
		{\it{Finite Fields}} (Encyclopedia of mathematics
		and its applications), 2nd ed. Cambridge, U.K., Cambridge Univ. Press,
		1997.
		
		\bibitem{Yip2022} C. H. Yip, ``Gauss sums and the maximum cliques in generalized Paley graphs of square order,'' \textit{Funct. Approx. Comment. Math.},  vol. 66, no. 1, pp. 119--138, 2022.	
		
	\end{thebibliography}
\end{document}